\documentclass[11pt, leqno, twoside]{article}

\usepackage{amssymb}
\usepackage{amsmath}
\usepackage{amsthm}
\usepackage{amsfonts}
\usepackage{mathrsfs}
\usepackage{indentfirst}
\usepackage{graphicx}
\usepackage{txfonts}

\usepackage{anysize}

\usepackage{fancyhdr}

\usepackage{color}
\usepackage{setspace}

\allowdisplaybreaks

\usepackage{txfonts}

\def\rr{{\mathbb R}}
\def\rn{{{\rr}^n}}

\def\cc{{\mathbb C}}
\def\nn{{\mathbb N}}

\def\cm{{\mathcal M}}

\def\ch{{\mathcal H}}

\def\fz{\infty}

\def\dz{\delta}

\def\lz{\lambda}

\def\lf{\left}
\def\r{\right}

\def\hs{\hspace{0.26cm}}

\def\ls{\lesssim}
\def\gs{\gtrsim}

\def\noz{\nonumber}
\def\wz{\widetilde}

\def\dist{\mathop\mathrm{\,dist\,}}

\def\loc{{\mathop\mathrm{\,loc\,}}}

\def\q1{\wz q}
\def\Q1{q_1}

\def\dint{\displaystyle\int}

\def\dist{{\mathop\mathrm{\,dist\,}}}
\def\loc{{\mathop\mathrm{loc\,}}}

\newtheorem{thm}{Theorem}[section]
\newtheorem{prop}[thm]{Proposition}
\newtheorem{lem}[thm]{Lemma}

\theoremstyle{definition}
\newtheorem{defn}[thm]{Definition}
\newtheorem{rem}[thm]{Remark}

\newtheorem{Example}[thm]{Example}

\numberwithin{equation}{section}

\begin{document}
\title{\bf\Large  Sharp Poincar\'e--Sobolev Inequalities of Choquet--Lorentz
Integrals with Respect to Hausdorff Contents on Bounded John Domains
\footnotetext{\hspace{-0.35cm}
2020 {\it Mathematics Subject Classification}.
{Primary 46E36; Secondary 42B35, 42B25, 26D10.} \endgraf
{\it Key words and phrases.} bounded John domain, Hausdorff content,
Choquet--Lorentz integral,
Poincar\'e(--Sobolev) inequality. \endgraf
This project is supported by the National Natural Science Foundation
of China (Grant Nos. 12201139, 12431006, 12371093, and 12371127), the National
Key Research and Development Program of China (Grant No. 2020YFA0712900), and the Natural Science Foundation of Hunan province (2024JJ3023).}}
\author{Long Huang, Yuanshou Cao, Dachun Yang and
Ciqiang Zhuo\footnote{Corresponding author, E-mail:
\texttt{cqzhuo87@hunnu.edu.cn}/{\color{red}{November 22, 2024}}/Final version.}
}
\date{ }
\maketitle
	
\vspace{-0.7cm}

\begin{center}
\begin{minipage}{13cm}
{\small {\bf Abstract}\quad
Let $\Omega$ be a bounded John domain in $\mathbb R^n$ with $n\ge 2$,
and let $\mathcal{H}_{\infty }^{\delta}$
denote the Hausdorff content of dimension $\delta\in (0,n]$.
In this article, the authors prove the Poincar\'e and the Poincar\'e--Sobolev
inequalities, with sharp ranges of indices, on Choquet--Lorentz integrals with respect to
$\mathcal{H}_{\infty }^{\delta}$ for all continuously differentiable
functions on $\Omega$. These results not only extend
the recent Poincar\'e and Poincar\'e--Sobolev inequalities
to the Choquet--Lorentz integrals, but also provide some endpoint
estimates (weak type) in the critical case. 		
One of the main novelties exists in that, to achieve the goals,
the authors develop some new tools associated with
Choquet--Lorentz integrals on $\mathcal{H}_{\infty }^{\delta}$, such as
the fractional Hardy--Littlewood maximal inequality and the Hedberg-type
pointwise estimate on the Riesz potential. As an application, the authors obtain
the sharp boundedness of the Riesz potential on Choquet--Lorentz integrals.
Moreover, even for classical Lorentz integrals,
these Poincar\'e and Poincar\'e--Sobolev
inequalities are also new.}
\end{minipage}
\end{center}

\vspace{0.2cm}

\section{Introduction}
	
On the $n$-dimensional Euclidean space $\rn$ with $n\ge 2$,
if $\Omega\subset \rn$ is a bounded Lipschitz domain or
even a John domain, then the $(p,p)$-Poincar\'{e} inequality
\begin{equation}\label{eq10-24-0}
\inf_{b\in\mathbb{R}}\lf[\int_{\Omega}|u(x)-b|^{p}\,dx\r]^\frac{1}{p}
\leq C\lf[\int_{\Omega}|\nabla u(x)|^{p}\,dx\r]^\frac 1p
\end{equation}
for any $u\in W^{1,p}(\Omega)$ with $p\in [1, \infty)$
and the $(\frac {np}{n-p},p)$-Poincar\'{e}--Sobolev inequality
\begin{equation}\label{eq10-24}	\inf_{b\in\mathbb{R}}
\lf[\int_{\Omega}|u(x)-b|^{\frac{np}{n-p}}\,dx
\r]^\frac{n-p}{np}
\leq C\lf[\int_{\Omega}|\nabla u(x)|^{p}\,dx\r]^\frac 1p
\end{equation}
for any $u\in W^{1,p}(\Omega)$ with $p\in [1, n)$ hold;
see, for instance, \cite{MGV1985}.
Here and thereafter, $ W^{1,p}(\Omega)$ denotes the classical
Sobolev space defined on $\Omega$ and,
for any differentiable function $u$ on $\rn$, $\nabla u$
denotes the gradient of $u$, that is,
$\nabla u:=(\frac{\partial u}{\partial x_1}, \ldots,
\frac{\partial u}{\partial x_n}).$
These Poincar\'e type inequalities are of great importance
in potential theory, harmonic analysis, and
partial differential equations; see, for instance, the book \cite{zsy16} and
the article \cite{dlyyz}.	
Recently, the counterparts of them in the setting of Choquet integrals with respect to the Hausdorff content $\ch_\fz^\delta$
of dimension $\delta\in(0,n]$ were obtained. To be precise, let $\Omega\subset\rn$, with $n\ge 2$,
be a bounded $(\alpha,\beta)$-John domain and $\delta\in(0,n]$. The
corresponding $(p,p)$-Poincar\'{e} inequality
\begin{align*}
\inf_{b\in\mathbb{R}}\int_{\Omega}|u(x)-b|^{p}
\,d\mathcal{H}_{\infty }^{\delta }\leq C\beta^p
\lf(\frac{\beta}{\alpha}\r)^{2np}\int_{\Omega}|\nabla u(x)|^{p}
\,d\mathcal{H}_{\infty }^{\delta }
\end{align*}
for any $u\in C^1(\Omega)$ (the set of
all continuously differentiable functions on $\Omega$) and $p\in(\frac{\delta}{n},\infty)$ and
the $(\frac{\delta p}{\delta-p},p)$-Poincar\'{e}--Sobolev inequality
\begin{equation}\label{eq10-24x}
\inf_{b\in\mathbb{R}}\lf\{\int_{\Omega}|u(x)-b|^{\frac{\delta p}{\delta-p}}
\,d\ch_\fz^\delta\r\}^\frac{\delta-p}{\delta p}
\leq C\lf\{\int_{\Omega}|\nabla u(x)|^{p}\,d\ch_\fz^\delta\r\}^\frac 1p
\end{equation}
for any $u\in C^1(\Omega)$ and $p\in(\frac {\delta}{n},\delta)$ hold; see \cite{hh23}.
These recover the aforementioned two classical inequalities by taking $\delta=n$.
	
Notice that, compared \eqref{eq10-24x}
with the classical Poincar\'e--Sobolev inequality \eqref{eq10-24},
the endpoint estimate for $p=\frac{\delta}{n}$  in the setting of Choquet
integrals is missing. Inspired by this problem, we are interested in the Poincar\'{e}--Sobolev
inequalities  based on Choquet integrals in the endpoint case $p=\frac{\delta}{n}$.
In this article, we first establish some weak type estimates in this case.
Furthermore,
we also study the Poincar\'{e}--Sobolev
inequalities in the context of
Choquet--Lorentz integrals with respect to the Hausdorff content
$\mathcal{H}_{\infty }^{\delta }$.
Recall that the Lorentz space, as an intermediate space of
Lebesgue spaces under real interpolation, traces back to
Lorentz \cite{am05,L51}. In view of their finer structures,
Lorentz spaces appear frequently in the study on various critical
or endpoint analysis problems from many different research fields.
For instance, Spector \cite{sp19} gave an
optimal Lorentz space estimate for the Riesz potential
acting on curl-free vectors and Spector and Van Schaftingen
\cite{sv19} established some optimal embeddings on the Lebesgue space $L^1$ into Lorentz spaces
for some vector differential operators via Gagliardo's lemma.
Indeed, there exist enormous literatures on
this subject, here we only mention several recent articles from
harmonic analysis (see, for instance, \cite{dhr09,ah15,Lyw2016}) and
partial differential equations (see, for instance, \cite{ps20,zyy22}).
	
In order to state our main theorems, we begin with recalling some
notation on Hausdorff contents and the associated Choquet--Lorentz integral.
Let $B(x,r):=\{y\in\rn:\ |y-x|<r\}$ denote the open ball
centered at $x\in\rn$ with the radius $r\in(0,\fz)$.
For the following definition of the Hausdorff
content of a set $K$ in $\rn$, we refer to the
survey \cite{su09} and also the book on the geometric measure theory
of Federer \cite[p.\,169]{f76}. For more progress
on functional inequalities about Hausdorff contents, we refer to Saito et al. \cite{st22,stw20,stw19,stw16}
 and also to \cite{actuv13,liu16,v86}.	
	
\begin{defn}\label{def-a1} Let $K$ be a set in $\rn$ with $n\geq2$
and let $\dz \in (0,n]$.
The \emph{Hausdorff content} $\mathcal{H}_{\infty }^{\delta }\lf (K \r )$
of $K$ is defined by setting
\begin{align}\label{0e1}
\mathcal{H}_{\infty }^{\delta }\lf (K \r )
:=\inf\lf \{\sum_i r_{i}^{\delta }:\
K\subset \bigcup_i B\lf ( x_{i},r_{i} \right ) \right \},
\end{align}
where the infimum is taken over all finite or countable ball coverings
$\{B(x_{i},r_{i})\}_{i}$ of $K$. The quantity $\mathcal{H}_{\infty }^{\delta }\lf (K \r )$
in \eqref{0e1} is also referred to as the $\delta$-\emph{Hausdorff content}
or the $\delta$-\emph{Hausdorff capacity} or the \emph{Hausdorff content of $K$  of
dimension} $\delta$.
\end{defn}
	
Let us now recall the definition of the Choquet integral with respect
to the Hausdorff content. In this article, $\Omega$ is always assumed to be a
\emph{domain} in $\rn$ with $n\ge 2$, that is, $\Omega$ is an open and connected
set of $\rn$. Let $\delta\in(0,n]$.
Then the \emph{Choquet integral with respect to the Hausdorff
content $\mathcal{H}_{\infty }^{\delta }$}, for short, the
\emph{Choquet integral} $\int_\Omega f(x)\,d\mathcal{H}_{\infty }^{\delta }$
of a non-negative function $f$ on $\Omega$ is defined by setting
\begin{align*}
\int_\Omega f(x)\,d\mathcal{H}_{\infty }^{\delta }:=\int_{0}^{\infty}
\mathcal{H}_{\infty }^{\delta }(\{x\in\Omega:\
f(x)>\lambda\})\,d\lambda.
\end{align*}
Since $\mathcal{H}_{\infty }^{n}$ is monotone, then,
for any function $g$ defined on $\Omega$, its distribution
function in the sense of the Hausdorff content
$\lambda\mapsto\mathcal{H}_{\infty }^{\delta }(\{x\in\Omega:\ |g(x)|>\lambda\})$
is decreasing on $\lambda\in[0,\fz)$.
Thus, we easily find that the above distribution function
is measurable with respect to the Lebesgue measure.
Therefore, $\int_{0}^{\infty}\mathcal{H}_{\infty }^{\delta }
(\{x\in\Omega:\ |g(x)|>\lambda\})\,d\lambda$ is a well-defined
Lebesgue integral and hence the Choquet integral is also well defined.
Based on this, for any $p\in(0,\fz)$, the \emph{$p$-Choquet
integral with respect to the Hausdorff content},
$\|f\|_{L^p(\Omega,\ch_\fz^\delta)}$,
of a function $f$ on $\Omega$ is defined by setting
\begin{align*}
\|f\|_{L^p(\Omega,\ch_\fz^\delta)}
:=\lf[\int_\Omega |f(x)|^p\,d\ch_\fz^\delta\r]^\frac1p
:=\lf[p\int_{0}^{\infty}{\lambda}^{p-1}
\mathcal{H}_{\infty }^{\delta }\lf(\lf\{x\in\Omega:\ |f(x)|>\lambda\r\}\r)
\,d\lambda\r]^{\frac{1}{p}}.
\end{align*}
In what follows, for any $p\in(0,\fz)$, we always denote
by $L^p(\Omega,\ch_\fz^\delta)$ the space of all functions
$f$ on $\Omega$ such that their quasi-norms $\|f\|_{L^p(\Omega,\ch_\fz^\delta)}$
are finite.
	
We point out that, compared with Riemann or Lebesgue integrals,
the Choquet integral with respect to the Hausdorff content
has the following significant differences:
\begin{enumerate}
\item[{\rm(i)}]
The Choquet integral is a nonlinear integral, that is, for any non-negative
functions $f$ and $g$ on $\Omega$,
$$\int_\Omega f(x)\,d\mathcal{H}_{\infty }^{\delta }+\int_\Omega g(x)\,d\mathcal{H}_{\infty }^{\delta }\neq\int_\Omega [f(x)+g(x)]\,d\mathcal{H}_{\infty }^{\delta }
\le 2\left[\int_\Omega f(x)\,d\mathcal{H}_{\infty }^{\delta }
+\int_\Omega g(x)\,d\mathcal{H}_{\infty }^{\delta }\right].$$
\item[{\rm(ii)}] Choquet integrals are well defined for all non-measurable functions
with respect to the Lebesgue measure.
\end{enumerate}
For more details on Choquet integrals, we refer to \cite{su09}.
The Choquet--Lorentz integral with respect to the Hausdorff content is
defined as follows.
	
\begin{defn}\label{def-a2}
Let $\delta\in(0, n]$, $p\in(0,\infty)$, and $q\in(0,\infty]$.
The \emph{Choquet--Lorentz integral with respect to the Hausdorff content
$\ch_\fz^\delta$}, for short, the \emph{Choquet--Lorentz integral}
$\|f\|_{L^{p,q}(\Omega,\mathcal{H}_{\infty }^{\delta })}$ of
a Lebesgue measurable function \emph{f} on $\Omega$ is defined by setting
\begin{align*}
\|f\|_{L^{p,q}(\Omega,\mathcal{H}_{\infty }^{\delta })}:=
\begin{cases}
\lf[p\dint_{0}^{\infty}\lambda^{q}
\lf\{\mathcal{H}_{\infty }^{\delta }\lf(\lf\{x\in\Omega:\ |f(x)|>\lambda\r\}\r)
\r\}^{\frac{q}{p}}\,\frac{d\lambda}
{\lambda}\r]^{\frac{1}{q}} &\rm{if}\ \emph{q}\in(0,\infty),\\
\sup\limits_{\lambda>0}\lambda\lf[\mathcal{H}_{\infty }^{\delta }
\lf(\lf\{x\in\Omega:\ |f(x)|>\lambda\r\}\r)\r]^{\frac{1}{p}}
&\rm{if}\ \emph{q}=\infty.
\end{cases}
\end{align*}
\end{defn}
	
For convenience, we use $L^{p,q}(\Omega,\mathcal{H}_{\infty }^{\delta })$
to denote the \emph{
Choquet--Lorentz space} of all functions $f$ on $\Omega$ having
finite quasi-norms $\|f\|_{L^{p,q}(\Omega,\mathcal{H}_{\infty }^{\delta })}$.
Indeed, the Choquet--Lorentz space has appeared in the survey \cite[Theorem 7]{su09}.
Moreover, from the definition, we infer that, when $p=q\in(0,\infty)$,
$\|\cdot\|_{L^{p,p}(\Omega,\mathcal{H}_{\infty }^{\delta })}
=\|\cdot\|_{L^{p}(\Omega,\mathcal{H}_{\infty }^{\delta })}$ and hence, in this case,
${L^{p,p}(\Omega,\mathcal{H}_{\infty }^{\delta })}=
L^{p}(\Omega,\mathcal{H}_{\infty }^{\delta });$
when $p\in(0,\infty)$ and $q=\infty$, the Choquet--Lorentz space
$L^{p,\infty}(\Omega,\mathcal{H}_{\infty }^{\delta })$
is the ``weak" version of $L^p(\Omega,\ch_\fz^\delta)$.

We still need the definition of the $(\alpha,\beta)$-John domain,
which was originally introduced by John in \cite{kr91}.

\begin{defn}\label{d1.3}
Assume that $\Omega$ is a bounded domain in $\rn$ with $n\geq2$.
Then $\Omega$ is called an \emph{$(\alpha,\beta)$-John domain}
if there exist constants $0<\alpha\leq\beta<\infty$ and
a point $x_0\in\Omega$ such that each point
$x\in\Omega$ can be joined to $x_0$ by
a rectifiable curve $\gamma_x:\ [0,\ell(\gamma_x)]\rightarrow\Omega$,
parameterized by its arc length, satisfying
$$\gamma_x(0)=x,\ \gamma_x(\ell(\gamma_x))
=x_0,\ \ell(\gamma_x)\leq\beta,$$
and
$$\dist\lf(\gamma_x(t),\partial\Omega\r)\geq\frac{\alpha}{\beta}t
\quad \forall\, t\in[0,l(\gamma_x)].$$
Here the point $x_0$ is called the \emph{John center} of $\Omega$.
\end{defn}
	
\begin{rem}
It is not difficult to show that convex domains and domains
with Lipschitz boundary are John domains.
The domains with fractal boundaries such as the von Koch snow
flake are also John domains,
but domains with outward spires are not allowed; see \cite{hh23}.
Moveover, let $\Omega:=B(\mathbf{0},k)\subset \rn$ with some positive integer $k$.
Here and thereafter, $\mathbf{0}$ denotes the \emph{origin} of $\rn$.
Taking $\alpha=\beta=k$ and
$x_0$ being the center of $\Omega$, then $\Omega$ is a $(k, k)$-\emph{John domain}.
\end{rem}

With these preparations, the first main theorem of this article is the following Poincar\'e inequality and endpoint
weak type estimate in terms of Choquet--Lorentz integrals with respect
to the Hausdorff content.
	
\begin{thm}\label{them-1}
Let $\Omega\subset\rn$ with $n\ge 2$ be a bounded $(\alpha,\beta)$-John domain
with $0<\alpha\leq\beta<\infty$ and let $\delta\in(0,n]$.
\begin{enumerate}
\item[{\rm(i)}]If $p,q\in(\frac{\delta}{n},\infty)$,
then there exists a positive constant $C$,
depending only on $n$, $\delta$, $p$, and $q$,
such that, for any $u\in C^1(\Omega)$,
\begin{align}\label{1ex1'}
\inf_{b\in\mathbb{R}}\|u-b\|_{L^{p,q}(\Omega,\mathcal{H}_{\infty }^{\delta })}
\leq C\beta\lf(\frac{\beta}{\alpha}\r)^{2n}\|\nabla u
\|_{L^{p,q}(\Omega,\mathcal{H}_{\infty}^{\delta})}.
\end{align}
\item[{\rm(ii)}] If $p=\frac \delta n$, then there exists
a positive constant $C$, depending only on $n$ and $\delta$,
such that, for any $u\in C^1(\Omega)$,
\begin{align}\label{1ex2'}
\inf_{b\in\rr}\|u-b\|_{L^{p,\fz}(\Omega,\ch_\fz^\delta)}
\le C\beta\lf(\frac{\beta}{\alpha}\r)^{2n}
\|\nabla u\|_{L^{p}(\Omega,\ch_\fz^\delta)}.
\end{align}
\end{enumerate}
\end{thm}
	
\begin{rem}\label{0r1'}
\begin{enumerate}
\item[(i)] We point out that we indeed prove two stronger
inequalities than those stated in Theorem \ref{them-1}, that is,
we show that the left-hand sides of \eqref{1ex1'} and \eqref{1ex2'}
can be enhanced, respectively, into
$\|u-u_B\|_{L^{p,q}(\Omega,\mathcal{H}_{\infty }^{\delta })}$ and
$\|u-u_B\|_{L^{p,\fz}(\Omega,\mathcal{H}_{\infty }^{\delta })},$
where
\begin{align}\label{1ex3'}
u_B:=\frac{1}{|B|}\int_{B}u(x)\,dx \ \
\text{with}\ \ B:=B\lf(x_0,C_{(n)}\frac{\alpha^2}{\beta}\r).
\end{align}
Here $x_0$ is the John center of the bounded $(\alpha,\beta)$-John
domain $\Omega$ and $C_{(n)}$ is a positive constant that depends only on $n$.
In addition, if $q=p\in(\frac{\delta}{n},\infty)$,
then the Choquet--Lorentz space $L^{p,q}(\Omega,\mathcal{H}_{\infty }^{\delta })
=L^{p}(\Omega,\mathcal{H}_{\infty }^{\delta })$.
Thus, as a special case, Theorem \ref{them-1}(i) with
$q=p\in(\frac{\delta}{n},\infty)$ reduces back to
\cite[Theorem 3.2]{hh23}. Particularly, when $\delta=n$,
we then obtain the classical $(p,p)$-Poincar\'e inequality \eqref{eq10-24-0}.

\item[(ii)] Notice that $\delta\le n$ and $p\in[\frac{\delta}{n},\infty)$.
Thus, the integrable exponent $p$ in Theorem \ref{them-1} may less than 1.
We point out that, when $p=\frac\delta n$, the endpoint weak type estimate proved
in Theorem \ref{them-1}(ii) completes \cite[Theorem 3.2]{hh23}
because the endpoint case $p=\frac \delta n$ in \cite[Theorem 3.2]{hh23} is excluded.
Moreover, the estimate \eqref{1ex2'} indeed
holds for all $p\in[\frac\delta n,\fz)$ by Theorem \ref{them-1}(i) (or
\cite[Theorem 3.2]{hh23}).

\item[(iii)] When $\delta=n$, we know that the Hausdorff
content $\mathcal{H}_{\infty }^{\delta }$ coincides with the Lebesgue measure
by \eqref{eq-0905a2}. Then the Poincar\'e inequalities in Theorem \ref{them-1}
with $q\neq p$ also give the Poincar\'e inequalities on Lorentz spaces with
respect to the Lebesgue measure. Even in this special case, Theorem \ref{them-1} seems
also new to the best of our knowledge.
\end{enumerate}
\end{rem}
	
As the second main theorem of this article, we establish the Poincar\'e--Sobolev inequality
and also the endpoint weak type estimate in terms of Choquet--Lorentz integrals
as follows, in which the dimension $\delta-\mu p$ of the Hausdorff content on
the left-hand side is less than or equal to the dimension $\delta$ on the right-hand side.
	
\begin{thm}\label{them-2}
Let $\Omega\subset\rn$ with $n\ge 2$ be a bounded $(\alpha,\beta)$-John domain
with $0<\alpha\leq\beta<\infty$.
Suppose that $\mu\in[0,1)$ and $\delta\in(0,n]$.
\begin{enumerate}
\item[{\rm(i)}] If $p\in(\frac{\delta}{n},\delta)$ and $q\in(\frac{\delta(\delta-\mu p)}{n(\delta-p)},\infty)$,
then there exists a positive constant $C$, depending only on $n$, $\mu$,
$\delta$, $p$, $q$, and the John constants $\alpha,\beta$,
such that, for any $u\in C^1(\Omega)$,
\begin{align}\label{1ex4'}
\inf_{b\in\mathbb{R}}\|u-b\|_{L^{\frac{p(\delta-\mu p)}{\delta-p},q}
(\Omega,\mathcal{H}_{\infty }^{\delta-\mu p })}
\leq C \|\nabla u\|_{L^{p,\frac{q(\delta-p)}{\delta-\mu p}}
(\Omega,\mathcal{H}_{\infty }^{\delta })}.
\end{align}
\item[{\rm(ii)}] If $p=\frac{\delta}{n}$, then there exists a positive constant
$C$,
depending only on $n$, $\mu$, $\delta$, and the John constants $\alpha,\beta$,
such that, for any $u\in C^1(\Omega)$,
\begin{align}\label{1ex5'}
\inf_{b\in\mathbb{R}}\|u-b\|_{L^{\frac{p(\delta-\mu p)}{\delta-p},\fz}
(\Omega,\mathcal{H}_{\infty }^{\delta-\mu p })}
\leq C \|\nabla u\|_{L^{p}(\Omega,\mathcal{H}_{\infty }^{\delta })}.
\end{align}
\end{enumerate}
\end{thm}

\begin{rem}\label{0r2'}
\begin{enumerate}
\item[(i)] We show that the exponent $\frac{p(\delta-\mu p)}{\delta-p}$
in Theorem \ref{them-2} is sharp by
a counterexample (see Remark \ref{ex-0927} below). Moreover, similarly
to Remark \ref{0r1'}(i), we indeed prove two stronger
inequalities than those stated in Theorem \ref{them-2}, that is,
the left-hand sides of \eqref{1ex4'} and \eqref{1ex5'}
can be enhanced, respectively, into
$\|u-u_B\|_{L^{\frac{p(\delta-\mu p)}{\delta-p},q}(\Omega,\mathcal{H}_{\infty }^{\delta-\mu p })}$
and
$\|u-u_B\|_{L^{\frac{p(\delta-\mu p)}{\delta-p},\fz}(\Omega,\mathcal{H}_{\infty }^{\delta-\mu p })}$
with the same $u_B$ as in \eqref{1ex3'}.

\item[(ii)] Theorem \ref{them-2} covers \cite[Theorem 3.7]{hh23}
and extends the range of exponents to the endpoint case $p=\frac \delta n$.
Indeed, by taking $q=\frac{p(\delta-\mu p)}{\delta-p}$ in
Theorem \ref{them-2}(i), we conclude that, for any $p\in(\frac{\delta}{n},\delta)$,
\begin{align}\label{1}
\inf_{b\in\mathbb{R}}\|u-b\|_{L^{\frac{p(\delta-\mu p)}{\delta-p}}
(\Omega,\mathcal{H}_{\infty }^{\delta-\mu p })}
\ls	\|\nabla u \|_{L^{p}(\Omega,\mathcal{H}_{\infty }^{\delta })}.
\end{align}
This inequality reduces back to \cite[Theorem 3.7]{hh23}.
Choosing $\mu=\frac 1p$ and $\delta=n$ in the last inequality, we obtain
$$
\inf_{b\in\mathbb{R}}\|u-b\|_{L^{\frac{p(n-1)}{n-p}}
(\Omega,\mathcal{H}_{\infty }^{n-1})}
\ls	\lf[\int_{\Omega}|\nabla u(x)|^p\,dx\r]^{\frac1p},
$$
which reduces back to \cite[Corollary 3.9]{hh23}.
In addition, when $\mu=0$, the estimate \eqref{1ex5'} is a strengthening
version of \eqref{1ex2'}
because $\frac{p\delta}{\delta-p}>p$ and $\Omega$ is a bounded domain.
			
\item[(iii)]
The interesting special case of Theorem \ref{them-2}(i)
with $q=\frac{p(\delta-\mu p)}{\delta-p}$ and $\mu=0$
reduces back to \cite[Corollary 3.7]{hh23} [see also \eqref{eq10-24x}].
Furthermore, when the dimension $\delta=n$,
we recover the classical $(\frac {np}{n-p},p)$-Poincar\'e--Sobolev inequality \eqref{eq10-24}.
In addition, the estimate \eqref{1ex5'} indeed
holds for all $p\in[\frac\delta n,\delta)$ by \eqref{1}.

\item[(iv)] Similarly to Remark \ref{0r1'}(iii), if $\delta=n$ and $\mu=0$,
the Poincar\'e--Sobolev inequality on Choquet--Lorentz integrals
in Theorem \ref{them-2}
with $q\neq \frac{p(\delta-\mu p)}{\delta-p}$ also gives the Poincar\'e--Sobolev
inequality on Lorentz spaces with respect to the Lebesgue measure.
Even in this special case, Theorem \ref{them-2} seems also new to the best
of our knowledge.
\end{enumerate}
\end{rem}

Recall that the mapping properties of Riesz potentials are always important in
harmonic analysis and potential theory. For instance, Hatano et al. \cite{hnsh23} obtained the boundedness of Riesz potentials on Bourgain-Morrey spaces; Nakai \cite{n01,n10,n14,n17} considered the mapping properties of Riesz potentials on Campanato spaces and Morrey spaces. Very recently, Alves et al. \cite{agt24} discovered a new bilinear Riesz potential for Euler--Riesz systems and established a uniform estimate.

As an application, we obtain the following boundedness of Riesz potentials
on Choquet--Lorentz integrals, which is an analogue of the Hardy--Littlewood--Sobolev
theorem on the fractional integral in the setting of  Choquet--Lorentz integrals with respect to the Hausdorff content. Let $\alpha\in(0,n)$.
Recall that the \emph{Riesz potential} $I_\alpha$ is defined by setting,
for any $f\in L^1_{\rm{loc}}(\rn)$ (the set of all locally integrable functions on $\rn$) and $x\in\rn$,
$$I_\alpha f(x):=\frac1{c_\alpha}\int_{\rn}\frac{f(y)}{|x-y|^{n-\alpha}}\,dy,$$
where $c_\alpha:=\pi^{n/2}2^{\alpha}\frac{\Gamma(\alpha/2)}{\Gamma((n-\alpha)/2)}$
and $\Gamma$ denotes the \emph{Gamma function}.

\begin{thm}\label{them-rp}
Let $\alpha\in(0,n)$, $\delta\in(0,n]$, and $\mu\in[0,\alpha)$.
\begin{enumerate}
\item[{\rm(i)}] If $p\in(\frac{\delta}{n},\frac{\delta}{\alpha})$
and $q\in(\frac{\delta(\delta-\mu p)}{n(\delta-p\alpha)},\infty)$,
then there exists a positive constant $C$, depending only on $n$,
$\alpha$, $\mu$, and $\delta$,
such that, for any $f\in L^1_{\rm{loc}}(\rn)$,
\begin{align*}
\|I_{\alpha}f\|_{L^{\frac{p(\delta-\mu p)}{\delta-p\alpha},q}
(\rn,\mathcal{H}_{\infty }^{\delta-\mu p})}
\leq C \|f\|_{L^{p,\frac{q(\delta-p\alpha)}{\delta-\mu p}}
(\rn,\mathcal{H}_{\infty }^{\delta })}.
\end{align*}
\item[{\rm(ii)}] If $p=\frac{\delta}{n}$, then there exists a positive constant
$C$, depending only on $n$,
$\alpha$, $\mu$, and $\delta$, such that, for any $f\in L^1_{\rm{loc}}(\rn)$,
\begin{align}\label{2}
\|I_{\alpha}f\|_{L^{\frac{p(\delta-\mu p)}{\delta-p\alpha},\fz}
(\rn,\mathcal{H}_{\infty }^{\delta-\mu p})}
\leq C \|f\|_{L^{p}(\rn,\mathcal{H}_{\infty }^{\delta })}.
\end{align}
\end{enumerate}
\end{thm}

\begin{rem}
\begin{enumerate}
\item[(i)] The exponent $\frac{p(\delta-\mu p)}{\delta-p\alpha}$
in Theorem \ref{them-rp} is sharp. Indeed, by constructing
a counterexample in Example \ref{rem0619} below, we show that the exponent is
the best possible.

\item[(ii)] The estimate \eqref{2} is actually true for any $p\in[\frac\delta n,\frac{\delta}{\alpha})$.
Indeed, by taking $q=\frac{p(\delta-\mu p)}{\delta-p\alpha}$ in
Theorem \ref{them-rp}(i), we conclude that, for any $p\in(\frac{\delta}{n},\frac{\delta}{\alpha})$,
\begin{align*}
\|I_{\alpha}f\|_{L^{\frac{p(\delta-\mu p)}{\delta-p\alpha}}
(\Omega,\mathcal{H}_{\infty }^{\delta-\mu p })}
\ls	\|f \|_{L^{p}(\Omega,\mathcal{H}_{\infty }^{\delta })}
\end{align*}
with the implicit positive constant independent of $f$.
This obviously implies \eqref{2} for any $p\in(\frac\delta n,\frac{\delta}{\alpha})$.
			
\item[(iii)] When $\delta=n$ and $\mu=0$,
Theorem \ref{them-rp} gives the boundedness of the Riesz potential
on Lorentz spaces with respect to the Lebesgue measure and,
furthermore, taking $q=\frac{p(\delta-\mu p)}{\delta-p\alpha}$ then
Theorem \ref{them-rp} reduces back to the Hardy--Littlewood--Sobolev
theorem on the fractional integral; see, for instance, Grafakos \cite[Theorem 1.2.3]{gbook}.
\end{enumerate}
\end{rem}

We establish our main results by borrowing some ideas from
the well-known method (see \cite{Boj88,Tr92}).
Indeed, we prove the desired Poincar\'e--Sobolev inequality with the
help of the boundedness of fractional Hardy--Littlewood maximal operators
and also the Hedberg-type pointwise estimate on the Riesz potential related
to Choquet integrals under consideration.
Recall that the classical Hedberg-type pointwise estimate on the Riesz potential
goes back to Hedberg \cite{HL1972} and the one related to
Choquet integrals with respect to the Hausdorff content was given
in \cite[Lemma 3.6]{hh23}.
The boundedness of the fractional Hardy--Littlewood maximal operator on Choquet integrals
with respect to the Hausdorff content
was proved by Adams \cite{su09} and simplified by Orobitg and Verdera \cite{hu03},
both heavily rely on the covering lemma. Later on,
Tang \cite{t11} extended these boundedness to the weighted Choquet space and
the Choquet--Morrey space by a similar method and then studied
the Carleson embeddings for weighted Sobolev spaces.
However, these two pivotal ingredients in the Choquet--Lorentz integral
setting were missing. Thus, to obtain the above main theorems,
we need to overcome these two essential difficulties.

To tackle these hurdles, we first notice that classical Lorentz spaces
are interpolation spaces of Lebesgue spaces, that is, the
real interpolation between Lebesgue spaces $L^1$ and $L^{\infty}$ gives
the family of classical Lorentz spaces
$L^{p,q}$ for any $p\in(1,\infty)$ and $q\in[1,\infty]$;
see \cite[p.\,300, Theorem 1.9]{bs1988}.
Based on this perspective, in Theorem \ref{thm0530} we establish
the boundedness of the fractional Hardy--Littlewood maximal operator
from
$L^{p,s}(\rn,\ch_\fz^{\delta})$ to $L^{p,r}(\rn,\ch_\fz^{\delta-\mu p})$ via
employing the equivalence of $\ch_\fz^{\delta}$ and the dyadic Hausdorff content
$\widetilde{\mathcal{H}}_{0}^{\delta }$ introduced by Yang and Yuan \cite{Ha96} and the real interpolation theorem
$$\lf(L^{p_0}(\rn, \widetilde{\ch}_0^\delta),L^{p_1}(\rn,\widetilde{\ch}_0^\delta)\r)_{\eta,q}
=L^{p,q}(\rn,\widetilde{\ch}_0^\delta)$$
from Cerd\`{a} et al. \cite{cms11}
(see Lemma \ref{prop0527} for the details)
rather than
beginning with the covering lemma as Adams \cite{su09},
Orobitg and Verdera \cite{hu03}, and Tang \cite{t11} did.
This fortunately works for the present case of Choquet--Lorentz integrals.
Indeed, this interpolation theorem combining
with the boundedness of fractional Hardy--Littlewood maximal operators
on Choquet integrals then immediately implies the
boundedness of fractional Hardy--Littlewood maximal operators
on Choquet--Lorentz spaces associate with Fatou capacity.
However, the Hausdorff content $\mathcal{H}_{\infty }^{\delta }$
is not a Fatou capacity. To overcome the deficiency of the Fatou property,
we use the dyadic Hausdorff content
$\widetilde{\mathcal{H}}_{0}^{\delta }$ from \cite{Ha96}
as a key bridge and finally establish Theorem \ref{thm0530}.
This hence overcomes the first difficulty.

On the other hand, applying
the fractional Hardy--Littlewood maximal operator,
we obtain the Hedberg-type pointwise
estimates on the Riesz potential in terms of Choquet--Lorentz
integrals under consideration; see Proposition \ref{lem-0908b1}.
The pivotal points of Proposition \ref{lem-0908b1} are to estimate
the Hausdorff context of the superlevel set
$\mathcal{H}_{\infty }^{\delta}(\{y\in\rn:\ |f(y)|>\lambda\})$
for any $\lambda\in(0,\fz)$
and then to
convert classical Lorentz integrals to Choquet--Lorentz integrals.
From the two crucial tools and  also the properties of
$(\alpha,\beta)$-John domains, we finally derive Theorems
\ref{them-1}, \ref{them-2}, and \ref{them-rp}.
Indeed, our proofs strongly depend on the structure of Choquet--Lorentz
integrals and the obtained fractional Hardy--Littlewood
maximal inequality as well as the new Hedberg-type pointwise
estimates.

The organization of the remainder of this article is as follows.

In Section \ref{s2}, the fractional Hardy--Littlewood maximal
inequality on Choquet--Lorentz integrals with respect to the Hausdorff
content is established. Section \ref{s2'} contains the Hedberg-type
pointwise estimates on the Riesz potential in terms of Choquet--Lorentz integrals.
The proofs of the main theorems and the corresponding Poincar\'e or Poincar\'e--Sobolev
inequalities for compactly supported continuously differentiable
functions defined on open connected sets  are given in Section \ref{s3}.

Finally, we make some convention on the notation. Let $\nn:=\{1,2,\dots\}$.
We denote by $C$ a positive constant which is independent of the main parameters involved,
but it may vary from line to line. Besides, we denote
$f\le Cg$ (resp. $f\ge Cg$) for a positive constant $C$ by
$f\lesssim g$ (resp. $f\gtrsim g$).
We write $f\sim  g$ if $f\lesssim g\lesssim f$.
Also, we denote by $\mathbf{1}_E$ the characteristic function of set $E\subset \rn$.

\section{The Fractional Hardy--Littlewood
Maximal Inequality}\label{s2}
	
Let $\mu \in[0,n)$. Recall that $L_{\loc}^1(\rn)$ denotes the set of all locally integrable functions on
$\rn$ and the \emph{fractional
Hardy--Littlewood maximal function}
$\cm_{\mu}(f)$ of $f\in L_{\loc}^1(\rn)$ with \emph{order}
$\mu$ is defined by setting, for any $x\in\rn$,
$$\cm_{\mu}(f)(x):=\sup_{r\in(0,\fz)}r^{\,\mu}\fint_{B(x,r)}|f(y)|\,dy,$$
where the barred integral denotes the integral average over the ball $B(x,r)$.
Notice that, when $\mu=0$, we denote $\cm_0(f)$ by $\cm(f)$, which is the
classical \emph{Hardy--Littlewood maximal function} of $f$.
	
The main result of this section is the following fractional
Hardy--Littlewood maximal inequality on the Choquent--Lorentz
integral.

\begin{thm}\label{thm0530}
Let $\delta\in(0,n]$, $\mu\in[0,n)$, $p\in(\frac\delta n,\frac\delta \mu)$,
$r\in(\frac\delta n,\fz)$, and $s\in(0,r]$.
Then the fractional Hardy--Littlewood maximal operator $\cm_\mu$ is bounded from
$L^{p,s}(\rn,\ch_\fz^{\delta})$ to $L^{p,r}(\rn,\ch_\fz^{\delta-\mu p})$,
that is, there exists a positive constant $C$ such that,
for any $f\in L^{p,r}(\rn,\ch_\fz^{\delta-\mu p})$,
\begin{equation*}
\|\cm_\mu(f)\|_{L^{p,r}(\rn,\ch_\fz^{\delta-\mu p})}
\le C\|f\|_{L^{p,s}(\rn,\ch_\fz^\delta)}.
\end{equation*}
\end{thm}

\begin{rem}
\begin{enumerate}
\item[(i)]
It is interesting that the above inequality allows the dimension of the Hausdorff content
on the left-hand side is less than that on the right-hand side, which leads the
fractional Hardy--Littlewood maximal operator $\cm_\mu$ being of $(p,p)$ type. This differs from the
classical $(p,q)$-type of $\cm_\mu$ on classical Lebesgue spaces with $\frac1q=\frac1p-\frac{\mu}{n}$.
\item[(ii)]
The conclusion of Theorem \ref{thm0530} is new even in the case $\mu=0$ and $r=s$.
That is, the Hardy--Littlewood maximal operator $\cm$ is bounded on $L^{p,r}(\rn,\ch_\fz^\delta)$
for any $p\in(\frac\delta n,\fz)$ and $r\in(\frac\delta n,\fz)$, which means that there exists a positive constant $C$ such that, for any $f\in
L^{p,r}(\rn,\ch_\fz^\delta)$,
\begin{equation*}
\|\cm(f)\|_{L^{p,r}(\rn,\ch_\fz^{\delta})}
\le C\|f\|_{L^{p,r}(\rn,\ch_\fz^\delta)}.
\end{equation*}
\item[(iii)] When $\delta=n$, $\mu=0$, and $r=s=p$, Theorem \ref{thm0530}
in this case reduces back to the classical Hardy--Littlewood maximal inequality.
\end{enumerate}
\end{rem}

We next show Theorem \ref{thm0530} by combining a real interpolation result of Choquet-integral spaces due to Cerd\`{a} et al. \cite[Corollary 1]{cms11} (see also Lemma \ref{prop0527}), the boundedness of $\cm_{\mu}$ from $L^p(\rn,\mathcal{H}_{\infty }^{\delta})$ to $L^p(\rn,\mathcal{H}_{\infty }^{\delta-\mu p})$ (see Lemma \ref{lem0531}), and also
the dyadic Hausdorff content $\widetilde{\mathcal{H}}_{0}^{\delta }$ introduced by Yang
and Yuan in \cite{Ha96}. For this purpose, we first collect
some known properties on the Choquet--Lorentz integral.
It is not difficult to obtain the following equivalent quasi-norms of
$\|\cdot\|_{L^{p,q}(\rn,\mathcal{H}_{\infty }^{\delta })}$; we omit the details.
	
\begin{lem}\label{lem0831a}
Let $\delta\in(0,n]$, $p\in(0,\fz)$, and $q\in(0,\fz]$.
If $f\in{L^{p,q}(\rn,\mathcal{H}_{\infty }^{\delta })}$, then
\begin{equation*}
\|f\|_{L^{p,q}(\rn,\mathcal{H}_{\infty }^{\delta })}\sim
\begin{cases}
\lf[\displaystyle\sum\limits_{i\in\mathbb{Z}}2^{iq}
\lf\{\mathcal{H}_{\infty }^{\delta }\lf(\lf\{x\in\rn:\ |f(x)|>2^i\r\}\r)
\r\}^{\frac{q}{p}}\r]^{\frac{1}{q}}
\  &\rm{if}\;$q$\,\in(0,\infty),\\
\sup\limits_{i\in\mathbb{Z}}2^i\lf[\mathcal{H}_{\infty }^{\delta }
\lf(\lf\{x\in\rn:\ |f(x)|>2^i\r\}\r)\r]^{\frac{1}{p}}\ &\rm{if}\;$q$\,=\infty,
\end{cases}
\end{equation*}
where the positive equivalence constants are independent of $f$.
\end{lem}

\begin{rem}\label{rem0530}
Let $\delta\in(0,n]$ and $p\in(0,\fz)$.
By Lemma \ref{lem0831a} and the well-known inequality that,
for any $\theta\in(0,1]$ and $\{a_i\}_{i=1}^\fz\subset \cc$,
\begin{align}\label{eq-0916a4}
\left(\sum_{i=1}^{\infty}|a_i|\right)^{\theta}\leq\sum_{i=1}^{\infty}{|a_i|}^{\theta},
\end{align}
we conclude that, if $0<s\le r<\fz$, then $L^{p,s}(\rn,\mathcal{H}_\fz^\delta)\subset L^{p,r}(\rn,\mathcal{H}_\fz^\delta)$
and, for any
$f\in L^{p,s}(\rn,\mathcal{H}_\fz^\delta)$,
$\|f\|_{L^{p,r}(\rn,\mathcal{H}_\fz^\delta)}\ls \|f\|_{L^{p,s}(\rn,\mathcal{H}_\fz^\delta)}$
with the implicit positive constant independent of $f$.
\end{rem}

From the definition of Choquet--Lorentz integrals, we easily infer the following conclusion; we
omit the details.

\begin{lem}\label{lem0909Lc1}
Let $\delta\in(0,n]$, $p\in(0,\infty)$, and $q\in(0,\infty]$. Then,
for any $f\in L^{p,q}(\rn,\mathcal{H}_{\infty }^{\delta})$
and $\nu\in(0,\infty)$, it holds that
\begin{align*}
\||f|^{\nu}\|_{L^{p,q}(\rn,\mathcal{H}_{\infty }^{\delta })}
=\|f\|^{\nu}_{L^{\nu p,\nu q}(\rn,\mathcal{H}_{\infty }^{\delta })}.
\end{align*}
\end{lem}

The following conclusion is a special case of \cite[Theorem 7(a)]{su09}.

\begin{lem}\label{lem0531}
Let $\delta\in(0,n]$ and $\mu\in[0,n)$.
If $p\in(\frac \delta n,\frac \delta\mu)$, then, for any
$f\in L^p(\rn,\ch_\fz^\delta)$,
$$\|\cm_\mu(f)\|_{L^p(\rn,\ch_\fz^{\delta-\mu p})}
\ls\|f\|_{L^p(\rn,\ch_\fz^\delta)}$$
and, if $p=\frac\delta n$, then
$$\sup_{\lz\in(0,\fz)}\lz \ch_\fz^{\delta-\mu p}
\lf(\lf\{x\in\rn:\ \cm_\mu(f)>\lz\r\}\r)^{1/p}
\ls \|f\|_{L^p(\rn,\ch_\fz^\delta)},$$
where the implicit positive constants are independent of $f$.
\end{lem}

To apply the real interpolation theorem, we need to recall more
properties on the Hausdorff content.
It is known that $\ch_\fz^\delta$ is an outer
capacity in the sense of Meyers \cite[p.\,257]{cruz03} and an outer measure.
Precisely, the Hausdorff content has the following properties (see, for example, \cite{fz01}).

\begin{lem}\label{rema1}
Let $\delta\in(0,n]$. Then the following statements hold.
\begin{enumerate}
\item[{\rm (\romannumeral1)}]$\mathcal{H}_{\infty }^{\delta }(\emptyset)=0$;
\item[{\rm(\romannumeral2)}] If $A\subset B\subset \rn$,
then $\mathcal{H}_{\infty }^{\delta }(A)\leq\mathcal{H}_{\infty }^{\delta }(B)$;
\item[{\rm(\romannumeral3)}]  If $K\subset \rn$, then
$$\mathcal{H}_{\infty }^{\delta }\lf(K\r)
=\inf \{\mathcal{H}_{\infty }^{\delta }\lf(U\r):\
U\supset K~\text{and}~\emph U~\text{is~open}\};$$
\item[{\rm (\romannumeral4)}]If $\{K_i\}_{i=1}^{\infty}$
is a decreasing sequence of compact sets in $\rn$, then
$$\mathcal{H}_{\infty }^{\delta }
\lf(\bigcap_{i=1}^{\infty}K_{i} \r)=\lim_{i\rightarrow\infty}
\mathcal{H}_{\infty }^{\delta }\lf(K_{i} \r);$$
\item[{\rm (\romannumeral5)}]If $\{E_i\}_{i=1}^{\infty}$
is any sequence of sets in $\rn$, then
$$\mathcal{H}_{\infty }^{\delta }\lf(\bigcup_{i=1}^{\infty}E_{i} \r)
\leq \sum_{i=1}^{\infty}
\mathcal{H}_{\infty }^{\delta }\lf(E_{i}\r).$$
\end{enumerate}
\end{lem}
	
However, the Hausdorff content $\mathcal{H}_{\infty }^{\delta }$ is not
a capacity in the sense of Choquet \cite{dh11} because the deficiency of the following property: for any increasing sequence $\{K_i\}_{i=1}^{\infty}$ of sets,
\begin{align*}
\mathcal{H}_{\infty }^{\delta }\lf(\bigcup_{i=1}^{\infty}K_{i} \r)
=\lim_{i\rightarrow\infty}
\mathcal{H}_{\infty }^{\delta }\lf(K_{i} \r).
\end{align*}
To overcome this obstacle, the \emph{dyadic Hausdorff content
$\widetilde{\mathcal{H}}_{\infty }^{\delta }$} was introduced
(see, for instance, \cite{fz01}), that is,
\begin{align*}
\widetilde{\mathcal{H}}_{\infty }^{\delta }\lf(K \r)
:=\inf\lf\{\sum_{i=1}^{\infty }[\emph{l}(Q_{i})]^{\delta }:\
K\subset\bigcup_{i=1}^{\infty}Q_{i} \r\},
\end{align*}
where the infimum is taken over all dyadic cube coverings of $K$
and where $l(Q)$ denotes the edge length of the cube $Q$.
This equivalent Hausdorff content was claimed to be a capacity in the sense of Choquet in \cite{fz01},
but it is true only when $\delta\in (n-1,n]$ (see \cite[Propositions 2.1 and 2.2]{Ha96}).
Indeed, Yang and Yuan in \cite[Proposition 2.2]{Ha96} proved that when
$\delta\in (0,n-1]$ there exist compact sets
$\{K_j\}_{j\in\nn}$ which is decreasing but
$$\widetilde{\mathcal{H}}_{\infty }^{\delta}\lf(\bigcap_{i=1}^\fz K_j\r)
\neq \lim_{j\to\fz} \widetilde{\mathcal{H}}_{\infty }^{\delta}(K_j).$$
To overcome this shortage, a \emph{new dyadic Hausdorff content} $\widetilde{\mathcal{H}}_{0}^{\delta }$
was introduced in \cite[Definition 2.1]{Ha96},
that is, for any subset $K$ of $\rn$,
\begin{align*}
\widetilde{\mathcal{H}}_{0}^{\delta }\lf(K \r)
:=\inf\lf\{\sum_i \lf[\emph{l}(Q_{i})\r]^{\delta }:\
K\subset\lf(\bigcup_i Q_{i}\r)^{\circ} \r\},
\end{align*}
where the infimum is taken over all dyadic cubes $\{Q_{i}\}_{i}$ and, for any $E\subset \rn$,
$E^{\circ}$ denote the \emph{interior} of the set $E$.
This dyadic Hausdorff content $\widetilde{\mathcal{H}}_{0}^{\delta }$ has
the following good properties.
	
\begin{rem}\label{rem0920b}
Let $\delta\in(0,n]$ with $n\ge 2$.

\begin{enumerate}
\item[(i)] By \cite[Proposition 2.3]{Ha96}, $\widetilde{\mathcal{H}}_{0}^{\delta}$
is equivalent to $\mathcal{H}_{\infty }^{\delta }$ with the positive
equivalence constants
depending only on $n$.

\item[(ii)]
It follows from \cite[Proposition 2.4 and Theorem 2.1]{Ha96} that
$\widetilde{\mathcal{H}}_{0}^{\delta}$ is strongly subadditive,
that is, for any sets $K_1,\,K_2 \subset\rn$,
\begin{align*}
\widetilde{\mathcal{H}}_{0}^{\delta}(K_1\cup K_2)
+\widetilde{\mathcal{H}}_{0}^{\delta}(K_1\cap K_2)
\leq\widetilde{\mathcal{H}}_{0}^{\delta}(K_1)+\widetilde{\mathcal{H}}_{0}^{\delta}(K_2)
\end{align*}
and, for any $\delta\in (0, n]$, it holds that
\begin{enumerate}
\item[{\rm(a)}]if $\{K_j\}_{j=1}^{\infty}$ are compact
and decrease to $K$, then $\lim_{j\rightarrow\infty}
\widetilde{\mathcal{H}}_{0}^{\delta}(K_j)=
\widetilde{\mathcal{H}}_{0}^{\delta}(K)$;
\item[{\rm(b)}]if $\{E_j\}_{j=1}^{\infty}$ increase to $E$,
then $\lim_{j\rightarrow\infty}\widetilde{\mathcal{H}}_{0}^{\delta}(E_j)=
\widetilde{\mathcal{H}}_{0}^{\delta}(E)$.
\end{enumerate}
Therefore, $\widetilde{\mathcal{H}}_{0}^{\delta}$ is a capacity in the sense of
 Choquet \cite{dh11} for all $\delta\in (0, n]$.
\end{enumerate}
\end{rem}

We next recall the concept of the real interpolation and, for general facts concerning the interpolation theory,
we refer to books \cite{bl76} and \cite{bk91}.
In \cite[Corollary 1]{cms11}, Cerd\`{a} et al. established a real interpolation result of
Choquet integrals corresponding to a capacity $C$, which satisfies:
\begin{enumerate}
\item[(i)] \emph{Quasi-subadditivity}: there exists $L\in[1,\fz)$ such that, for any subsets $A$ and $B$ of $\rn$,
$$C(A\cup B)\le L[C(A)+C(B)];$$
\item[(ii)] \emph{Fatou property}: for a sequence $\{E_j\}_{j=1}^\fz$ of subsets in $\rn$,
if $\{E_j\}_{j=1}^\fz$ increase to $E$, then
$$\lim_{j\to\fz} C(E_j)=C(E).$$
\end{enumerate}

We point out that the dyadic Hausdorff content $\widetilde{\ch}_0^\delta$ satisfies the
quasi-subadditivity and the Fatou property for any $\delta\in(0,n]$ by Remark \ref{rem0920b}(ii).
But, the Hausdorff content $\mathcal{H}_{\infty }^{\delta }$ does not have the Fatou property.
Therefore, we can only apply \cite[Corollary 1]{cms11} to the Choquet integral
with respect to the dyadic Hausdorff content $\widetilde{\ch}_0^\delta$. Fortunately,
$\widetilde{\mathcal{H}}_{0}^{\delta}$ is equivalent to $\mathcal{H}_{\infty }^{\delta }$
due to \cite[Proposition 2.3]{Ha96}.

Let $p_0,p_1\in(0,\fz)$. Then the space
$L^{p_0}(\rn,\widetilde{\ch}_0^\delta)+L^{p_1}(\rn,\widetilde{\ch}_0^\delta)$
is defined to be the set of all measurable functions $f$
on $\rn$ such that $f=f_0+f_1$, where $f_0\in
L^{p_0}(\rn,\widetilde{\ch}_0^\delta)$ and
$f_1\in
L^{p_1}(\rn,\widetilde{\ch}_0^\delta)$.
For simplicity of presentation, let
$A_0:=L^{p_0}(\rn,\widetilde{\ch}_0^\delta)$  and
$A_1:=L^{p_1}(\rn,\widetilde{\ch}_0^\delta).$
Then the \emph{real interpolation space} $(A_0,A_1)_{\theta,q}$, with $\theta\in(0,1)$
and $q\in(0,\fz]$, is defined to be the set
of all functions $f\in A_0+A_1$ satisfying
$$\|f\|_{\theta,q}:=\lf\{\int_0^\fz \lf[t^{-\theta}K(t,f;A_0,A_1)\r]^q\,\frac{dt}t\r\}^{\frac1q}<\fz,$$
where $K(t,f;A_0,A_1)$ is called the \emph{$K$-functional},
defined by setting
$$K(t,f;A_0,A_1):=\inf\lf\{\|f_0\|_{A_0}+t\|f_1\|_{A_1}:\ f=f_0+f_1,~~f_0\in A_0,~~f_1\in A_1\r\}.$$

Now, we state the real interpolation result of the Choquet integral corresponding to
the dyadic Hausdorff content $\widetilde{\ch}_0^\delta$ as follows,
which is a special case of \cite[Corollary 1]{cms11}.

\begin{lem}\label{prop0527}
Let $\delta\in(0,n]$ and $p\in(0,\fz)$. Then, for any $0<p_0<p_1<\fz$, $q\in(p_0,\fz)$, and $\eta\in(0,1)$ with
$\frac1p=\frac{1-\eta}{p_0}+\frac{\eta}{p_1}$,
$(L^{p_0}(\rn, \widetilde{\ch}_0^\delta),L^{p_1}(\rn,\widetilde{\ch}_0^\delta))_{\eta,q}
=L^{p,q}(\rn,\widetilde{\ch}_0^\delta).$
\end{lem}

We are now ready to prove Theorem \ref{thm0530}.

\begin{proof}[Proof of Theorem \ref{thm0530}]
Let $\delta\in(0,n]$, $p\in(\frac\delta n,\frac\delta \mu)$,
$r\in(\frac\delta n,\fz)$, and $s\in(0,r]$. We choose $\eta\in(0,1)$, $p_0,p_1\in (\frac\delta n,\frac\delta \mu)$ such that $p_0<r$,
$p_0<p_1$, and $\frac1p=\frac{\eta}{p_0}+\frac{1-\eta}{p_1}$. Then, by Lemma \ref{prop0527}, we conclude that,
for any $\delta\in(0,n]$,
\begin{equation}\label{eq0531c}
L^{p,r}(\rn,\widetilde{\ch}_0^\delta)=\lf(L^{p_0}(\rn, \widetilde{\ch}_0^{\delta}),
L^{p_1}(\rn,\widetilde{\ch}_0^{\delta})\r)_{\eta,r}
\end{equation}
and, moreover, for any $\mu\in [0,n)$,
\begin{equation}\label{eq0531}
L^{p,r}(\rn,\widetilde{\ch}_0^{\delta-\mu p})
=\lf(L^{p_0}(\rn, \widetilde{\ch}_0^{\delta-\mu p}),L^{p_1}(\rn,\widetilde{\ch}_0^{\delta-\mu p})\r)_{\eta,r}.
\end{equation}

Let $f\in L^{p,s}(\rn,\ch_\fz^\delta)$. Then $f\in L^{p,r}(\rn,\widetilde{\ch}_0^\delta)$
due to Remarks \ref{rem0920b}(i) and \ref{rem0530}.
Thus, by \eqref{eq0531c}, we find that
 there exist two functions $f_0\in L^{p_0}(\rn, \widetilde{\ch}_0^\delta)$ and
$f_1\in L^{p_1}(\rn,\widetilde{\ch}_0^\delta)$ such that $f=f_0+f_1$ and
\begin{align}\label{eq0531a}
\lf(\int_0^\fz\lf\{t^{-\eta}\lf[\|f_0\|_{L^{p_0}(\rn, \widetilde{\ch}_0^\delta)}+
t\|f_1\|_{L^{p_1}(\rn, \widetilde{\ch}_0^\delta)}\r]\r\}^r\,\frac{dt}t\r)^{\frac1r}
\ls \|f\|_{L^{p,r}(\rn,\widetilde{\ch}_0^\delta)}.
\end{align}

Next, we let
$E_0:=\{x\in\rn:\ \frac12\cm_\mu (f)(x)\le \cm_\mu (f_0)(x)\}$ and
$$E_1:=\lf\{x\in\rn:\ \frac12\cm_\mu (f)(x)\le \cm_\mu (f_1)(x)\r\}.$$
Then, by an obvious inequality that $\cm_\mu(f)(x)\le \cm_\mu(f_0)(x)+\cm_{\mu}(f_1)(x)$ for any $x\in\rn$,
we have
$$\rn=E_0\cup E_1=E_0\cup(E_1\backslash E_0).$$
Then, using $p_0,p_1\in (\frac\delta n,\frac\delta \mu)$, Lemma \ref{lem0531}, and definitions of $E_0$ and $E_1$, we
conclude that
\begin{equation*}
\cm_\mu(f)=\cm_\mu(f)\mathbf{1}_{E_0}+\cm_\mu(f)\mathbf{1}_{E_1\backslash E_0}\in
L^{p_0}(\rn, \widetilde{\ch}_0^{\delta-\mu p})+L^{p_1}(\rn, \widetilde{\ch}_0^{\delta-\mu p}).
\end{equation*}
From this, Remark \ref{rem0920b}(i), \eqref{eq0531}, Lemma \ref{lem0531} again, \eqref{eq0531a}, and Remark \ref{rem0530}, we deduce that
\begin{align*}
&\|\cm_\mu(f)\|_{L^{p,r}(\rn,\ch_\fz^{\delta-\mu p})}\\
&\quad\sim \|\cm_\mu(f)\|_{L^{p,r}(\rn,\widetilde{\ch}_0^{\delta-\mu p})}
\sim\|\cm_\mu(f)\|_{(L^{p_0}(\rn, \widetilde{\ch}_0^{\delta-\mu p}),
L^{p_1}(\rn,\widetilde{\ch}_0^{\delta-\mu p}))_{\eta,r}}\\
&\quad\ls \lf\{\int_0^\fz \lf(t^{-\eta}\lf[\|\cm_\mu(f)\mathbf{1}_{E_0}\|_{L^{p_0}(\rn, \ch_\fz^{\delta-\mu p})}
+t\|\cm_\mu(f)\mathbf{1}_{E_1\backslash E_0}\|_{L^{p_1}(\rn, \ch_\fz^{\delta-\mu p})}\r]\r)^r\,\frac{dt}t\r\}^\frac1r\\
&\quad \ls  \lf\{\int_0^\fz \lf(t^{-\eta}\lf[\|f_0\|_{L^{p_0}(\rn, \widetilde{\ch}_0^\delta)}
+t\|f_1\|_{L^{p_1}(\rn, \widetilde{\ch}_0^\delta)}\r]\r)^r\,\frac{dt}t\r\}^\frac1r\\
&\quad \ls \|f\|_{L^{p,r}(\rn,\widetilde{\ch}_0^\delta)}\ls \|f\|_{L^{p,r}(\rn,\ch_\fz^\delta)}\ls \|f\|_{L^{p,s}(\rn,\ch_\fz^\delta)}.
\end{align*}
This finishes the proof of Theorem \ref{thm0530}.
\end{proof}

\section{Hedberg-Type Pointwise Estimates on Riesz Potentials}\label{s2'}

In this section, we establish some Hedberg-type pointwise estimates on Riesz potentials
in terms of Choquet--Lorentz integrals by borrowing some ideas from \cite{ad75,hh23}.
These estimates play an important role in the proofs of Theorems \ref{them-2} and \ref{them-rp}.

\begin{prop}\label{lem-0908b1}
Let $\alpha\in(0,n)$, $\mu\in[0,\alpha)$, and $\delta\in(0,n]$.
\begin{enumerate}
\item[{\rm(i)}] If $p\in(\frac{\delta}{n},\frac{\delta}{\alpha})$ and $q\in(0,\infty]$,
then there exists a positive constant $C_1$ such that,
for any $f\in L^1_{\rm{loc}}(\rn)$ and $x\in\rn$,
$$\int_{\rn}\frac{|f(y)|}{|x-y|^{n-\alpha}}\,dy
\leq C_1\lf[\cm_{\mu}f(x)\r]^{\frac{\delta-p\alpha}{\delta-\mu p}}
\lf\|f\r\|^{\frac{p(\alpha-\mu)}{\delta-\mu p}}
_{L^{p,\frac{q(\delta-p\alpha)}{\delta-\mu p}}(\rn,\mathcal{H}_{\infty}^{\delta})}.$$
\item[{\rm(ii)}] If $p=\frac{\delta}{n}$, then there exists a
positive constant $C_2$ such that, for any $f\in L^1_{\rm{loc}}(\rn)$ and $x\in\rn$,
$$\int_{\rn}\frac{|f(y)|}{|x-y|^{n-\alpha}}\,dy
\leq C_2\lf[\cm_{\mu}f(x)\r]^{\frac{\delta-p\alpha}{\delta-\mu p}}
\lf\|f\r\|^{\frac{p(\alpha-\mu)}{\delta-\mu p}}
_{L^{p}(\rn,\mathcal{H}_{\infty}^{\delta})}.$$
\end{enumerate}
\end{prop}

To prove Proposition \ref{lem-0908b1}, notice that,
for any Lebesgue measurable function $f$ on $\Omega$,
it holds
\begin{equation}\label{eq0530d}
\int_{\Omega}|f(x)|\,dx
\ls
\lf[\int_{\Omega}|f(x)|^{\frac \delta n}\,d\mathcal{H}_{\infty}^{\delta}
\r]^{\frac n \delta},
\end{equation}
where $\Omega\subset \rn$ is a domain, $n\ge 2$, and $\delta\in(0,n]$ (see \cite[Lemma 3]{hu03}).
Recall that the classical Lorentz space $L^{p,q}(\rn)$
is originally defined via the decreasing rearrangement
and further characterized by distribution functions
(see, for instance, \cite[Proposition 1.4.9]{ll02}).
Let $p\in(0,\infty)$ and $q\in(0,\infty]$. Then $f\in L^{p,q}(\rn)$
if and only if
\begin{align}\label{eq-0916a2}
\|f\|_{L^{p,q}(\rn)}:=
\lf\{\begin{array}{ll}
p^{\frac{1}{q}}\lf[\displaystyle\int_{0}^{\infty}
\lf\{\lambda |\{x\in\rn:\ |f(x)|>\lambda\}|^{\frac{1}{p}}\r\}^{q}\,\frac{d\lambda}
{\lambda}\r]^{\frac{1}{q}}&\rm{if}\ \emph{q}\in(0,\infty),\\
\sup\limits_{\lambda\in(0,\fz)}{\lambda}|\{x\in\rn:\ |f(x)|>\lambda\}|^{\frac{1}{p}}
&\rm{if}\ \emph{q}=\infty
\end{array}
\r.
\end{align}
is finite.

We now show Proposition \ref{lem-0908b1}.

\begin{proof}[Proof of Proposition \ref{lem-0908b1}]
Let $p\in[\frac{\delta}{n},\frac{\delta}{\alpha})$,
$q\in(0,\infty]$, $f\in L^1_{\rm{loc}}(\rn)$, and $x\in\rn$. For any $r\in(0,\fz)$, we write
\begin{equation}\label{eq-0922c}
\int_{\rn}\frac{|f(y)|}{|x-y|^{n-\alpha}}\,dy
=\lf\{\int_{B(x,r)}+
\int_{\rn \backslash B(x,r)}\r\}\frac{|f(y)|}{|x-y|^{n-\alpha}}\,dy.
\end{equation}
For the first term of the right-hand side of \eqref{eq-0922c}, by
the assumption that $\mu<\alpha$,
we find that
\begin{align}\label{eq-0908a2}
\int_{B(x,r)}\frac{|f(y)|}{|x-y|^{n-\alpha}}\,dy
&=\sum_{k=1}^{\infty}\int_{\{x\in\rn:\ 2^{-k}r
\le |x-y|<2^{-k+1}r\}}\frac{|f(y)|}{|x-y|^{n-\alpha}}\,dy \\
&\leq\sum_{k=1}^{\infty}\lf(\frac{2^{k}}{r}\r)^{n-\alpha}
\int_{B\lf(x,2^{-k+1}r\r)}|f(y)|\,dy
\ls r^{\alpha-\mu}\cm_{\mu}f(x).\nonumber
\end{align}

We first prove (i) in the following two steps.

\textbf{Step 1.} We show (i) for any $q\in(0,\fz)$.
To achieve this, let $\widetilde{p}:=\frac{p(\delta-\mu p)}{\delta-p\alpha}$.
Next, we estimate the second term of the right-hand side of
\eqref{eq-0922c} by considering the following two cases on $q$.

\textbf{Case 1.} $\,q\in(\frac{\widetilde p}{2p},\fz)$. In this case, $\frac{2pq}{\widetilde{p}}\in(1,\fz)$.
Then, by the fact $p>\frac{\delta}n$ and the H\"{o}lder inequality of Lorentz spaces
(see \cite[Theorem 1.4.17(vi)]{ll02}),
we conclude that
\begin{align*}
\int_{\rn\setminus{B(x,r)}}\frac{|f(y)|}{|x-y|^{n-\alpha}}\,dy
&\leq\|f\|_{L^{\frac{pn}{\delta},\frac{2pq}{\widetilde{p}}}(\rn\setminus{B(x,r)})}
\lf\|{|x-\cdot|^{-n+\alpha}}
\r\|_{L^{p^{\prime}_{1},q^{\prime}_{1}}(\rn\setminus{B(x,r)})},
\end{align*}
where $p^{\prime}_{1}$ and $q^{\prime}_{1}$ denote the conjugate
indices of $\frac{pn}{\delta}$ and $\frac{2pq}{\widetilde{p}}$, respectively.
Since $\mathcal{H}_{\infty}^{n}(K)\leq(\mathcal{H}_{\infty }^{\delta}(K))
^{\frac{n}{\delta}}$ for any set $K\subset \rn$ by both
\eqref{eq-0916a4} and $\delta\in(0,n]$, it follows from \eqref{eq-0905a2} that
\begin{align}\label{eq-0922d}
\|f\|_{L^{\frac{pn}{\delta},\frac{2pq}{\widetilde{p}}}(\rn\setminus{B(x,r)})}
&\sim\lf[\int_{0}^{\infty}{\lambda}^{\frac{2pq}{\widetilde{p}}}\lf|\lf\{y\in\rn
\setminus{B(x,r)}:\
\lf|f(y)\r|>\lambda\r\}\r|^{\frac{2q\delta}{\widetilde{p}n}}\,
\frac{d\lambda}{\lambda}\r]^{\frac{\widetilde{p}}{2pq}} \\
&\lesssim\lf\{\int_{0}^{\infty}{\lambda}^{\frac{2pq}{\widetilde{p}}}
\lf[\mathcal{H}_{\infty }^{n}\lf(\lf\{y\in\rn:\ \lf|f(y)\r|>\lambda\r\}\r)\r]^
{\frac{2q\delta}{\widetilde{p}n}}\,\frac{d\lambda}{\lambda}
\right\}^{\frac{\widetilde{p}}{2pq}} \nonumber\\
&\ls \lf\{\int_{0}^{\infty}{\lambda}^{\frac{2pq}{\widetilde{p}}}
\lf[\mathcal{H}_{\infty }^{\delta}\lf(\{y\in\rn:\ \lf|f(y)\r|>\lambda\}\r)\r]
^{\frac{2q}{\widetilde{p}}}\,\frac{d\lambda}{\lambda}\r\}
^{\frac{\widetilde{p}}{2pq}}.\nonumber
\end{align}
For any $\lambda\in(0,\fz)$, let
$h(\lambda):=\mathcal{H}_{\infty }^{\delta}(\{y\in\rn:\ |f(y)|>\lambda\})$.
Then the function $h$ is decreasing. Notice that $\frac{pq}{\widetilde{p}}
=\frac{q(\delta-p\alpha)}{\delta-\mu p}>0$.
Thus, we have
\begin{align*}
{\lambda}^{\frac{pq}{\widetilde{p}}}h(\lambda)^{\frac{q}{\widetilde{p}}}
&\sim\int_{0}^{\lambda}{t}^{\frac{pq}{\widetilde{p}}}h(\lambda)
^{\frac{q}{\widetilde{p}}}\,\frac{dt}{t}
\ls \int_{0}^{\lambda}{t}^{\frac{pq}{\widetilde{p}}}h(t)^{\frac{q}{\widetilde{p}}}
\,\frac{dt}{t} \\
&\ls \int_{0}^{\infty}{t}^{\frac{pq}{\widetilde{p}}}h(t)^{\frac{q}{\widetilde{p}}}
\,\frac{dt}{t}
\sim \|f\|^{\frac{pq}{\widetilde{p}}}_{L^{p,\frac{pq}{\widetilde{p}}}
(\rn,\mathcal{H}_{\infty }^{\delta })}.
\end{align*}
From this and \eqref{eq-0922d}, we infer that
\begin{align}\label{eq-0909Lc1}
\|f\|_{L^{\frac{pn}{\delta},\frac{2pq}{\widetilde{p}}}(\rn\setminus{B(x,r)})}
\ls \lf[\|f\|^{\frac{pq}{\widetilde{p}}}_{L^{p,\frac{pq}{\widetilde{p}}}
(\rn,\mathcal{H}_{\infty }^{\delta })}
\int_{0}^{\infty}{\lambda}^{\frac{pq}{\widetilde{p}}}
h(\lambda)^{\frac{q}{\widetilde{p}}}\,\frac{d\lambda}{\lambda}
\r]^{\frac{\widetilde{p}}{2pq}}
\sim \|f\|_{L^{p,\frac{pq}{\widetilde{p}}}(\rn,\mathcal{H}_{\infty }^{\delta })}.\end{align}
	
Now, we turn to estimate $\|{|x-\cdot|^{-n+\alpha}}\|_{L^{p^{\prime}_{1},
q^{\prime}_{1}}(\rn\setminus{B(x,r)})}$
for any given $x\in\rn$. By \eqref{eq-0916a2}, we obtain
\begin{align*}
\lf\|{|x-\cdot|^{-n+\alpha}}\r\|_{L^{p^{\prime}_{1},q^{\prime}_{1}}(\rn\setminus{B(x,r)})}
&\sim \left(\int_{0}^{\frac{1}{r^{n-\alpha}}}{\lambda}^{q^{\prime}_{1}}
\left|\left\{y\in\rn:\ r\le |x-y|<\left(\frac{1}{\lambda}
\right)^{\frac{1}{n-\alpha}}\right\}\right|^{\frac{q^{\prime}_{1}}
{p^{\prime}_{1}}}\,\frac{d\lambda}{\lambda} \right)^{\frac{1}{q^{\prime}_{1}}}\\		
&\lesssim\left[\int_{0}^{\frac{1}{r^{n-\alpha}}}{\lambda}^{q^{\prime}_{1}
-\frac{nq^{\prime}_{1}}{(n-\alpha)p^{\prime}_{1}}}
\,\frac{d\lambda}{\lambda} \right]^{\frac{1}{q^{\prime}_{1}}}\noz
\sim r^{-(n-\alpha)[1-\frac n{(n-\alpha)p^{\prime}_1}]}\sim r^{\alpha-\frac{\delta}{p}},\noz
\end{align*}
where we used the fact that $p<\frac{\delta}{\alpha}$ and hence $1-\frac{n}{(n-\alpha)p^{\prime}_{1}}>0$. Therefore,
in the case where $q>\frac{\widetilde{p}}{2p}$, it holds
\begin{align}\label{eq-1122}
\int_{\rn\setminus B(x,r)}\frac{|f(y)|}{|x-y|^{n-\alpha}}\,dy
\ls r^{\alpha-\frac \delta p}\lf\|f\r\|_{L^{p,\frac{pq}{\widetilde{p}}}
(\rn,\mathcal{H}_{\infty}^{\delta})}.
\end{align}
	
\textbf{Case 2.} $\,q\in(0,\frac{\widetilde{p}}{2p}]$. In this case, $\frac{2pq}{\widetilde{p}}\in(0,1]$.
From the H\"older inequality of Lorentz spaces in \cite[Theorem 1.4.17(v)]{ll02},
we deduce that
\begin{align*}
\int_{\rn\setminus{B(x,r)}}\frac{|f(y)|}{|x-y|^{n-\alpha}}\,dy
\ls\|f\|_{L^{\frac {pn}{\delta},\frac{2pq}{\widetilde{p}}}(\rn\setminus{B(x,r)})}
\lf\|{|x-\cdot|^{-n+\alpha}}\r\|_{L^{p^{\prime}_{1},\fz}(\rn\setminus{B(x,r)})},
\end{align*}
where $p^{\prime}_{1}$ still denotes the conjugate index of $\frac{pn}{\delta}$.
Applying an argument similar to that used in \eqref{eq-0909Lc1}, we find that
$$\|f\|_{L^{\frac{pn}{\delta},\frac{2pq}{\widetilde{p}}}(\rn\setminus{B(x,r)})}
\ls\|f\|_{L^{p,\frac{pq}{\widetilde{p}}}(\rn,\mathcal{H}_{\infty }^{\delta })}.$$		
Moreover, we also have
\begin{align*}	\lf\|{|x-\cdot|^{-n+\alpha}}
\r\|_{L^{p^{\prime}_{1},\fz}(\rn\setminus{B(x,r)})}
&=\sup_{\lambda\in(0,\fz)}\lambda\lf|\lf\{y\in\rn\setminus{B(x,r)}:\
|y-x|<\lf(\frac 1\lambda\r)^{\frac{1}{n-\alpha}}\r\}\r|^{\frac1{p^{\prime}_{1}}}\\			
&\leq\sup_{\lambda\in(0,\frac 1{r^{n-\alpha}})}\lambda\lf|\lf\{y\in\rn:\
|y-x|<\lf(\frac 1\lambda\r)^{\frac{1}{n-\alpha}}\r\}\r|^{\frac1{p^{\prime}_{1}}}
\ls r^{\alpha-\frac{\delta}{p}}.
\end{align*}
Thus, in this case, we also obtain the estimate \eqref{eq-1122}. Altogether,
we conclude that, for any $q\in(0,\fz)$,
\begin{equation}\label{eq11-17a}
\int_{\rn\setminus{B(x,r)}}\frac{|f(y)|}{|x-y|^{n-\alpha}}\,dy
\ls r^{\alpha-\frac{\delta}{p}}\lf\|f\r\|_{L^{p,\frac{pq}{\widetilde{p}}}
(\rn,\mathcal{H}_{\infty }^{\delta })}.
\end{equation}
	
Combining \eqref{eq-0922c}, \eqref{eq-0908a2}, and \eqref{eq11-17a}, we find that
\begin{align*}
\int_{\rn}\frac{|f(y)|}{|x-y|^{n-\alpha}}\,dy
\ls r^{\alpha-\mu}\cm_{\mu}f(x)+r^{\alpha-\frac{\delta}{p}}
\lf\|f\r\|_{L^{p,\frac {pq}{\widetilde{p}}}(\rn,\mathcal{H}_{\infty }^{\delta })}.
\end{align*}
By taking $r:=[\frac{\cm_{\mu}f(x)}{\|f\|_{L^{p,\frac {pq}{\widetilde{p}}}
(\rn,\mathcal{H}_{\infty }^{\delta })}}]^{-\frac{p}{\delta-\mu p}},$
we further obtain, for any $x\in\rn$,
\begin{align*}
\int_{\rn}\frac{|f(y)|}{|x-y|^{n-\alpha}}\,dy
\ls\lf[\cm_{\mu}f(x)\r]^{\frac{\delta-p\alpha}{\delta-\mu p}}
\|f\|^{\frac{p(\alpha-\mu)}{\delta-\mu p}}
_{L^{p,\frac{q(\delta-p\alpha)}{\delta-\mu p}}(\rn,\mathcal{H}_{\infty }^{\delta })}.
\end{align*}
This finishes the proof of (i)
for any $q\in(0,\infty)$.

\textbf{Step 2.} We show (i) for $q=\fz$. Applying the H\"{o}lder inequality
of Lorentz spaces
(see \cite[p.73]{ll02}), we find that
\begin{align*}
\int_{\rn\setminus{B(x,r)}}\frac{|f(y)|}{|x-y|^{n-\alpha}}\,dy
&\ls\|f\|_{L^{\frac{pn}{\delta},\fz}(\rn\setminus{B(x,r)})}
\lf\|{|x-\cdot|^{-n+\alpha}}\r\|_{L^{\frac{pn}{pn-\delta},1}(\rn\setminus{B(x,r)})}.
\end{align*}
Furthermore, we have
\begin{align*}
\|f\|_{L^{\frac{pn}{\delta},\fz}(\rn\setminus{B(x,r)})}
&=\sup_{\lambda\in(0,\fz)}{\lambda}|\{x\in\rn\setminus{B(x,r)}:\
|f(x)|>\lambda\}|^{\frac{\delta}{pn}}\\
&\ls\sup_{\lambda\in(0,\fz)}{\lambda}[\mathcal{H}_{\infty }
^{\delta}(\{x\in\rn:\ |f(x)|>\lambda\})]^{\frac{1}{p}}
\sim\|f\|_{L^{p,\fz}(\rn,\mathcal{H}_{\infty }^{\delta})}
\end{align*}
and $$\lf\|{|x-\cdot|^{-n+\alpha}}\r\|_{L^{\frac{pn}{pn-\delta},1}
(\rn\setminus{B(x,r)})}\ls r^{\alpha-\frac\delta p},$$
which, together with \eqref{eq-0922c} and \eqref{eq-0908a2}, further imply that
\begin{align*}
\int_{\rn}\frac{|f(y)|}{|x-y|^{n-\alpha}}\,dy
\ls r^{\alpha-\mu}\cm_{\mu}f(x)+r^{\alpha-\frac{\delta}{p}}
\lf\|f\r\|_{L^{p,\fz}(\rn,\mathcal{H}_{\infty }^{\delta })}.
\end{align*}
Letting $r:=[\frac{\cm_{\mu}f(x)}{\|f\|_{L^{p,\fz}
(\rn,\mathcal{H}_{\infty }^{\delta })}}]^{-\frac{p}{\delta-\mu p}},$
we further obtain, for any $x\in\rn$,
\begin{align*}
\int_{\rn}\frac{|f(y)|}{|x-y|^{n-\alpha}}\,dy
\ls\lf[\cm_{\mu}f(x)\r]^{\frac{\delta-p\alpha}{\delta-\mu p}}
\|f\|^{\frac{p(\alpha-\mu)}{\delta-\mu p}}
_{L^{p,\fz}(\rn,\mathcal{H}_{\infty }^{\delta })}.
\end{align*}
This finishes the proof of (i)
for $q=\infty$ and hence (i).
	
Next, we prove (ii). Let $p=\frac{\delta}{n}$. In this case, we first have
\begin{align}\label{eq-0927a}
\int_{\rn\setminus{B(x,r)}}\frac{|f(y)|}{|x-y|^{n-\alpha}}\,dy
\leq r^{\alpha-n}\int_{\rn\setminus{B(x,r)}}|f(y)|\,dy.
\end{align}
By \eqref{eq0530d}, we obtain
\begin{align}\label{eq-1016a}
\int_{\rn\setminus{B(x,r)}}|f(y)|\,dy
\ls\lf[\int_{\rn}|f(y)|
^{\frac \delta n}\,d\mathcal{H}_{\infty}^{\delta}\r]^{\frac n \delta}
\sim\|f\|_{L^{p}(\rn,\mathcal{H}_{\infty }^{\delta})}.
\end{align}
Combining \eqref{eq-0922c}, \eqref{eq-0908a2}, \eqref{eq-0927a}, and \eqref{eq-1016a},
we conclude that
$$
\int_{\rn}\frac{|f(y)|}{|x-y|^{n-\alpha}}\,dy
\ls r^{\alpha-\mu}\cm_{\mu}f(x)+r^{\alpha-\frac{\delta}{p}}
\|f\|_{L^{p}(\rn,\mathcal{H}_{\infty }^{\delta})}.
$$
By taking $r:=[\frac{\cm_{\mu}f(x)}{\|f\|_{L^{p}(\rn,
\mathcal{H}_{\infty }^{\delta })}}]^{-\frac p{\delta-\mu p}},$
we find that
$$\int_{\rn}\frac{|f(y)|}{|x-y|^{n-\alpha}}\,dy
\ls\lf[\cm_{\mu}f(x)\r]^{\frac{\delta-p \alpha}{\delta-\mu p}}
\lf\|f\r\|^{\frac{p(\alpha-\mu)}{\delta-\mu p}}
_{L^{p}(\rn,\mathcal{H}_{\infty}^{\delta})}.$$
This finishes the proof of (ii)
and hence Proposition \ref{lem-0908b1}.
\end{proof}

\begin{rem} We point out that
Proposition \ref{lem-0908b1}(i) with $q=\frac{p(\delta-\mu p)}{\delta-p\alpha}$
and $\alpha=1$ reduces back to \cite[Lemma 3.6]{hh23}. In this sense,
Proposition \ref{lem-0908b1} extends \cite[Lemma 3.6]{hh23} to Choquet--Lorentz
integrals and to all $\alpha\in(0,n)$. Moreover,
the estimate in Proposition \ref{lem-0908b1}(ii) is actually true for any $p\in[\frac\delta n,\frac{\delta}{\alpha})$ by
combining Proposition \ref{lem-0908b1}(i).
\end{rem}

\section{Proofs of Theorems \ref{them-1}, \ref{them-2}, and \ref{them-rp}}\label{s3}

In this section, we prove Theorems \ref{them-1}, \ref{them-2}, and \ref{them-rp}.
For this purpose, let $\delta\in(0, n]$. Recall that the \emph{$\delta$-dimensional Hausdorff measure
of $K\subset\rn$} is defined by setting
\begin{align*}
\mathcal{H}^{\delta}(K):=\lim_{\varrho\rightarrow0^+}\inf
\left\{\sum_{i} r_{i}^{\delta}:\
K\subset\bigcup_{i}B(x_i,r_i),
\,r_i\leq\varrho~ \rm{for~any}~\emph{i}~\right\},
\end{align*}
where the infimum is taken over all such finite or countable ball
covers $\{B(x_{i},r_{i})\}_{i}$ of $K$ that $r_{i}\leq\varrho$ for any $i$.
When $\delta=n$, there exists a
positive constant $C$, depending only on $n$, such that, for any
Lebesgue measurable set $K\subset\rn$,
\begin{align}\label{eq-0905a1}
C^{-1}\mathcal{H}^{n}\left(K\right)
\leq|K|\leq C\mathcal{H}^{n}\left(K \right).
\end{align}
Here $|K|$ denotes the \emph{Lebesgue measure} of $K$.
For more properties of the Hausdorff measure we refer
to \cite[Chapter 2]{EG1992}. In particular, we have the following relation between
the Hausdorff content and the Hausdorff measure when $\delta=n$
(see \cite[Proposition 2.5]{hh23}):
There exists a positive constant $C$, depending only on $n$,
such that, for any $K\subset\rn$,
$\mathcal{H}_{\infty }^{n}(K)\leq\mathcal{H}^{n}(K)
\leq C\mathcal{H}_{\infty }^{n}(K).$
Combining this and \eqref{eq-0905a1}, we easily conclude that
there exists a positive constant $C$, depending only on $n$,
such that, for any Lebesgue measurable set $K\subset\rn$,
\begin{align}\label{eq-0905a2}
C^{-1}\mathcal{H}_{\infty }^{n}(K)\leq|K|\leq C\mathcal{H}_{\infty }^{n}(K).
\end{align}
Thus, when $\delta=n$, the Hausdorff content $\ch_\fz^\delta$ is
equivalent to the Lebesgue measure.
In this case, the space $L^{p,q}(\rn, \ch_\fz^\delta)$
coincides with
the classical Lorentz space $L^{p,q}(\rn)$.

The boundedness of $\cm$ on $L^p(\rn,\mathcal{H}_{\infty }^{\delta})$
is presented below, which is exactly \cite[Theorem]{hu03}.

\begin{lem}\label{a1.2}
Let $\delta\in(0,n]$. If $p\in(\frac \delta n,\fz)$, then, for any $f\in L^p(\rn,\ch_\fz^\delta)$,
it holds
$$\|\cm(f)\|_{L^p(\rn,\ch_\fz^\delta)}\ls \|f\|_{L^p(\rn,\ch_\fz^\delta)}$$
and, if $p=\frac \delta n$, then
\begin{align*}
\sup_{\lambda\in(0,\fz)}\lambda{\mathcal{H}}_{\infty }^{\delta}\lf(\lf\{x\in\rn:\
\cm (f)(x)>\lambda\r\}\r)^{1/p}
\ls \|f\|_{L^p(\rn,\ch_\fz^\delta)},
\end{align*}
where the implicit positive constants are independent of $f$.
\end{lem}

As an analogue of \eqref{eq0530d} for Lorentz spaces,
we have the following conclusion.
\begin{lem}\label{lem-1018b}
Let $\Omega\subset\rn$ with $n\geq2$ be a domain, $\delta\in(0,n]$,
$p\in(0,\infty)$, and $q\in(0,\infty]$.
Then, for any Lebesgue measurable function $f$,
\begin{equation}\label{eq-1202a}
\|f\|_{L^{p,q}(\Omega)}\ls\|f\|_{L^{\frac {p\delta} n,q}
(\Omega,\mathcal{H}_{\infty}^{\delta})}
\end{equation}
with the implicit positive constant independent of $f$ and $\Omega$.
\end{lem}
\begin{proof}
By \eqref{eq-0916a4},
we find that, for any Lebesgue measurable set $E$ of $\rn$,
$\ch_\fz^n(E) \le [\ch_\fz^\delta(E)]^\frac n\delta,$
which, together with \eqref{eq-0905a2}, further implies that
$|E|\sim \ch_\fz^n(E) \ls [\ch_\fz^\delta(E)]^\frac n\delta.$
From this, together with the definitions of
$\|f\|_{L^{p,q}(\Omega)}$ and $\|f\|_{L^{\frac {p\delta} n,q}
(\Omega,\mathcal{H}_{\infty}^{\delta})}$, we further infer that
\eqref{eq-1202a} holds. This finishes the proof of Lemma \ref{lem-1018b}.
\end{proof}

We now prove Theorem \ref{them-1}.

\begin{proof}[Proof of Theorem \ref{them-1}]
To show (i), let $p,q\in(\frac{\delta}{n},\infty)$. We may assume
$\|\nabla u\|_{L^{p,q}(\Omega,\mathcal{H}_{\infty}^{\delta})}$ is finite.
Since $\frac{pn}{\delta}>1$, it follows from
the H\"older inequality of Lorentz spaces
(see \cite[(v) and (vi) of Theorem 1.4.17 and p.\,73]{ll02}) and Lemma \ref{lem-1018b}
that
\begin{align}\label{3}
\int_{\Omega}|\nabla u|\,dx
\ls\|\nabla u\|_{L^{\frac{pn}{\delta},q}(\Omega)}\|\mathbf{1}_{\Omega}
\|_{L^{({\frac{pn}{\delta}})^\prime,q^\prime}(\Omega)}
\ls\|\nabla u\|_{L^{p,q}(\Omega,\mathcal{H}_{\infty}^{\delta})},
\end{align}
where we used the assumption that the domain $\Omega$ is bounded
in the last inequality.
This further implies
$|\nabla u|\in L^1(\Omega)$ and hence both the Riesz potential and the maximal function of
$|\nabla u|$ are well defined.
	
By \cite[Theorem]{ns12} (see also \cite{Boj88,Tr92}), we find that,
for any given $(\alpha,\beta)$-John domain $\Omega$ with the John center $x_0\in\rn$
and for any $u\in C^1(\Omega)$,
\begin{align}\label{eq-0909Lc3'}
|u(x)-u_B|\le C_{(n)}\lf(\frac{\beta}{\alpha}\r)^{2n}\int_{\Omega}
\frac{|\nabla u(y)|}{|x-y|^{n-1}}\,dy,\quad \forall\,x\in \Omega,
\end{align}
where $C_{(n)}$ is a positive constant depending only on $n$,
$$B:=B\lf(x_0,C_{(n)}\frac{\alpha^2}{\beta}\r),~~~\text{and}~~~u_B:=\frac{1}{|B|}\int_{B}u(x)\,dx.$$
Moreover, by \cite[Lemma 2.8.3]{zsy16} (see also \cite[Lemma]{HL1972}),
we conclude that the Riesz potential
on the bounded domain $\Omega$ can be estimated by the Hardy--Littlewood
maximal operator, which further implies that, for any $u\in C^1(\Omega)$
and $x\in \Omega$,
\begin{equation}\label{eq11-17'}
|u(x)-u_B|\ls\lf(\frac{\beta}{\alpha}\r)^{2n}{\rm diam}
(\Omega)\cm|\nabla u|(x).
\end{equation}
With the understanding that $|\nabla u|$ is zero outside $\Omega$,
combining \eqref{eq11-17'} and Theorem \ref{thm0530} with $\mu=0$, we conclude that
\begin{align*}
\|u-u_B\|_{L^{p,q}(\Omega,\mathcal{H}_{\infty }^{\delta })}
&\ls\lf(\frac{\beta}{\alpha}\r)^{2n}{\rm diam}(\Omega)
\lf\{\int_{0}^{\infty}\lambda^{q}\lf[\mathcal{H}_{\infty }^{\delta }
(\{x\in\Omega:\ \cm|\nabla u|(x)>\lambda\})\r]^{\frac{q}{p}}
\,\frac{d\lambda}{\lambda}\r\}^{\frac{1}{q}}\\
&\sim\lf(\frac{\beta}{\alpha}\r)^{2n}{\rm diam}(\Omega)~
\|\cm|\nabla u|\|_{L^{p,q}(\Omega,\mathcal{H}_{\infty }^{\delta })}
\ls\beta\lf(\frac{\beta}{\alpha}\r)^{2n}
\|\nabla u\|_{L^{p,q}(\Omega,\mathcal{H}_{\infty }^{\delta })},
\end{align*}
where we used ${\rm diam}(\Omega) \le 2\beta$ in the last inequality.
This finishes the proof of (i).

For (ii), similarly to the proof of (i), we also find that
$|\nabla u|$ is well defined.
If $p=\frac \delta{n}$, by \eqref{eq11-17'} and
the weak boundedness of
$\cm$ in  Lemma \ref{a1.2}, we obtain
\begin{align*}
\|u-u_B\|_{L^{p,\fz}(\Omega,\mathcal{H}_{\infty }^{\delta })}
&\ls\lf(\frac{\beta}{\alpha}\r)^{2n}{\rm diam}(\Omega)
\sup_{\lambda\in(0,\fz)}\lambda\lf[\mathcal{H}_{\infty }^{\delta }
(\{x\in\Omega:\ \cm|\nabla u|(x)>\lambda\})\r]^{\frac{1}{p}}\\
&\ls\lf(\frac{\beta}{\alpha}\r)^{2n}{\rm diam}(\Omega)
\|\nabla u\|_{L^{p}(\Omega,\ch_\fz^\delta)}
\ls\beta \lf(\frac{\beta}{\alpha}\r)^{2n}
\|\nabla u\|_{L^{p}(\Omega,\ch_\fz^\delta)}.
\end{align*}
Therefore, (ii) is true, which completes the proof of Theorem \ref{them-1}.
\end{proof}

We next show Theorem \ref{them-2}.

\begin{proof}[Proof of Theorem \ref{them-2}]
Let $\mu\in[0,1)$ and $\delta\in(0,n]$. We first prove (i).
To achieve this, suppose $p\in(\frac{\delta}{n},\delta)$ and
$q\in(\frac{\delta(\delta-\mu p)}{n(\delta-p)},\infty)$.
Then, applying \eqref{3},
we also obtain $|\nabla u|\in L^1(\Omega)$. From \eqref{eq-0909Lc3'}
and Proposition \ref{lem-0908b1}(i) with the
understanding that $|\nabla u|$ is zero outside $\Omega$ and $\alpha=1$,
we infer that, for any $x\in \Omega$,
\begin{align*}
|u(x)-u_B|
\le C_{(n)}\lf(\frac{\beta}{\alpha}\r)^{2n}\int_{\Omega}
\frac{|\nabla u(y)|}{|x-y|^{n-1}}\,dy
&\ls\left[\cm_{\mu}|\nabla u|(x)\right]^{\frac{\delta-p}{\delta-\mu p}}
\|\nabla u\|^{\frac{p(1-\mu)}
{\delta-\mu p}}_{L^{p,\frac{q(\delta-p)}{\delta-\mu p}}
(\Omega,\mathcal{H}_{\infty }^{\delta })},
\end{align*}
where $B$ is the same as in \eqref{eq-0909Lc3'}. Thus, by
Lemma \ref{lem0909Lc1} and Theorem \ref{thm0530} with the fact that $\frac{q(\delta-p)}{\delta-\mu p}>\frac{\delta}n$, we find that
\begin{align*}
&\lf\|u-u_B\r\|_{L^{\frac{p(\delta-\mu p)}{\delta-p},q}(\Omega,
\mathcal{H}_{\infty }^{\delta-\mu p})}\\
&\hs\ls \|\nabla u\|^{\frac{p(1-\mu)}{\delta-\mu p}}
_{L^{p,\frac{q(\delta-p)}{\delta-\mu p}}(\Omega,\mathcal{H}_{\infty
}^{\delta })}\lf\|\lf[\cm_{\mu}|\nabla u|\r]^{\frac{\delta-p}
{\delta-\mu p}}\r\|_{L^{\frac{p(\delta-\mu p)}{\delta-
p},q}(\Omega,\mathcal{H}_{\infty }^{\delta-\mu p })} \\
&\hs\sim \|\nabla u\|^{\frac{p(1-\mu)}{\delta-\mu p}}
_{L^{p,\frac{q(\delta-p)}{\delta-\mu p}}
(\Omega,\mathcal{H}_{\infty }^{\delta })}
\lf\|\cm_{\mu}|\nabla u|\r\|^{\frac{\delta-p}{\delta-\mu
p}}_{L^{p,\frac{q(\delta-p)}{\delta-\mu p}}(\Omega,\mathcal{H}
_{\infty }^{\delta-\mu p })} \\
&\hs\ls\|\nabla u\|^{\frac{p(1-\mu)}{\delta-\mu p}}
_{L^{p,\frac{q(\delta-p)}{\delta-\mu p}}(\Omega,\mathcal{H}
_{\infty }^{\delta })}\|\nabla u\|^{\frac{\delta-p}{\delta-\mu p}}
_{L^{p,\frac{q(\delta-p)}{\delta-\mu p}}(\Omega,\mathcal{H}_{\infty }^{\delta })}
\sim\|\nabla u\|_{L^{p,\frac{q(\delta-p)}{\delta-\mu p}}
(\Omega,\mathcal{H}_{\infty }^{\delta })},
\end{align*}
which completes the proof of (i).

For (ii), let $p=\frac{\delta}{n}$.
By \eqref{eq-0909Lc3'} and Proposition \ref{lem-0908b1}(ii) with the
understanding that $|\nabla u|$ is zero outside $\Omega$ and $\alpha=1$,
 we conclude that, for any $x\in \Omega$,
\begin{align*}
|u(x)-u_B|
\le C_{(n)}\lf(\frac{\beta}{\alpha}\r)^{2n}\int_{\Omega}
\frac{|\nabla u(y)|}{|x-y|^{n-1}}\,dy
&\ls \left[\cm_{\mu}(|\nabla u|)(x)\right]^{\frac{\delta-p}{\delta-\mu p}}
\|\nabla u\|^{\frac{p(1-\mu)}
{\delta-\mu p}}_{L^{p}
(\Omega,\mathcal{H}_{\infty }^{\delta })},
\end{align*}
which, combined with Lemma \ref{lem0909Lc1} and the weak boundedness of
$\cm_{\mu}$ in Lemma \ref{lem0531}, further implies that
\begin{align*}
\lf\|u-u_B\r\|_{L^{\frac{p(\delta-\mu p)}{\delta-p},\fz}(\Omega,
\mathcal{H}_{\infty }^{\delta-\mu p})}
&\hs\ls \|\nabla u\|^{\frac{p(1-\mu)}
{\delta-\mu p}}_{L^{p}
(\Omega,\mathcal{H}_{\infty }^{\delta })}
\lf\|\cm_{\mu}(|\nabla u|)\r\|^{\frac{\delta-p}{\delta-\mu p}}_{L^{p,\fz}(\Omega,\mathcal{H}
_{\infty }^{\delta-\mu p})} \\
&\hs\ls\|\nabla u\|^{\frac{p(1-\mu)}
{\delta-\mu p}}_{L^{p}
(\Omega,\mathcal{H}_{\infty }^{\delta })}\|\nabla u\|^{\frac{\delta-p}{\delta-\mu p}}
_{L^{p}(\Omega,\mathcal{H}_{\infty }^{\delta })}
\sim\|\nabla u\|_{L^{p}
(\Omega,\mathcal{H}_{\infty }^{\delta})}.
\end{align*}
Thus, conclusion (ii) is proved.
This then finishes the proof of Theorem \ref{them-2}.
\end{proof}

We would like to mention that, after this article is near finishing, we find that Harjulehto and
Hurri-Syrj\"{a}nen \cite[Theorem 3.6]{hharx} and
\cite[Theorem 4.11]{hharx1} have also obtained
an analog of Theorem \ref{them-2}(ii),
while our article concerns quite different topics
from \cite{hharx,hharx1};
hence, these articles cannot cover each other and
are of independent interest.

We now prove that the exponent $\frac{p(\delta-\mu p)}{\delta-p}$ in
Theorem \ref{them-2} is the best possible.
To show this, we give the following example.

\begin{Example}\label{ex-0927}	
Let $\Omega:=B(\mathbf{0},1)\backslash\{\mathbf{0}\}$ and $u_0(x):=|x|^{\eta}\textbf{1}_{\Omega}(x)$ for any $x\in\rn,$
where $\eta\in(-\fz,0)$ is chosen later. Then $u_0\in C^{1}(\Omega)$.
Let $\delta\in(0,n]$,
$\mu\in[0,1)$, and $p\in(\frac{\delta} n,\delta)$.
We claim that, if $s\in(\frac{p(\delta-\mu p)}{\delta-p},\fz)$,
then there exists $\eta$ such that
$\| u_0-b\|_{L^{s,q}(\Omega,\mathcal{H}_{\infty }^{\delta-\mu p})}
=\infty$ for any $b\in\rr$ and $q\in(0,\fz)$,	
but
\begin{equation}\label{eq0619a}
\|\nabla u_0\|_{L^{p,\widetilde{q}}(\Omega,\mathcal{H}_{\infty }^{\delta })}
<\infty\quad {\rm for\ any}\quad \widetilde{q}\in(0,\fz).
\end{equation}
	
Given any $b\in\rr$,  since $\eta<0$, it follows that there exists
$r\in(0,1)$ such that $\frac{1}{2} u_0(x)>b$ for any $x\in \rn$ satisfying
$|x|\in(0,r)$.
Then we obtain
\begin{align*}
\| u_0-b\|_{L^{s,q}(\Omega,\mathcal{H}_{\infty }^{\delta-\mu p})}^q
&\gs\int_0^{\infty}{\lambda}^q\lf[\mathcal{H}_{\infty }^{\delta-\mu p}
\lf(\lf\{x\in B(\mathbf{0},r):\ u_0(x)>2\lambda\r\}\r)\r]
^{\frac{q}{s}}\,\frac{d\lambda}{\lambda}\\
&\sim\int_0^{\infty}{\lambda}^q\lf[\mathcal{H}_{\infty }
^{\delta-\mu p}\lf(\lf\{x\in B(\mathbf{0},r):\ |x|<(2\lambda)^\frac1{\eta}\r\}\r)\r]
^{\frac{q}{s}}\,\frac{d\lambda}{\lambda}\noz\\
&\gs\int_{\frac{r^{\eta}}{2}}^{\infty}{\lambda}^q
\lf[\mathcal{H}_{\infty }^{\delta-\mu
p}\lf(B\lf(\mathbf{0},(2\lambda)^\frac1{\eta}\r)\r)\r]^{\frac{q}{s}}\,\frac{d\lambda}{\lambda}
\sim\int_{\frac{r^{\eta}}{2}}^{\infty}{\lambda}
^{q+\frac{(\delta-\mu p)q}{\eta s}}\,\frac{d\lambda}{\lambda}.\noz
\end{align*}
The last integral is infinite when $q+\frac{(\delta-\mu p)q}{\eta s}\geq0$,
that is, $\eta\leq-\frac{\delta-\mu p}{s}$.
	
On the other hand, observing that $|\nabla u_0(x)|=|\eta||x|^{\eta-1}$
for any $x\in \Omega$, we find that
\begin{align}\label{eq-0924b}
\|\nabla u_0\|_{L^{p,\widetilde{q}}(\Omega,\mathcal{H}_{\infty }
^{\delta })}^{\widetilde{q}}
&\sim \int_0^{\infty}{\lambda}^{\widetilde{q}}
\lf[\mathcal{H}_{\infty }^{\delta}
\lf(\lf\{x\in\Omega:\ |\eta||x|^{\eta-1}>\lambda\r\}\r)\r]
^{\frac{\widetilde{q}}p}\,\frac{d\lambda}{\lambda}\\
&\ls \int_0^{1}{\lambda}^{\widetilde{q}}\lf[\mathcal{H}_{\infty }
^{\delta}(B(\mathbf{0},1))\r]
^{\frac{\widetilde{q}}{p}}\,\frac{d\lambda}{\lambda}
+\int_1^{\infty}{\lambda}^{\widetilde{q}}
\lf[\mathcal{H}_{\infty }^{\delta}\lf(B\lf(\mathbf{0},c\lambda^{\frac1{\eta-1}}\r)\r)
\r]^{\frac{\widetilde{q}}p}\,
\frac{d\lambda}{\lambda},\noz
\end{align}
where $c:={\frac{1}{|\eta|^{1/(\eta-1)}}}$. By the fact that $\widetilde{q}>0$ and
$\mathcal{H}_{\infty }^{\delta}(B(\mathbf{0},r))\le r^{\delta}$,
we conclude that the first term of the right-hand side of
\eqref{eq-0924b} is finite.
For the second term of the right-hand side of \eqref{eq-0924b}, we have
\begin{align*}
\int_1^{\infty}{\lambda}^{\widetilde{q}}
\lf[\mathcal{H}_{\infty }^{\delta}
\lf(B\lf(\mathbf{0},c\lambda^{\frac1{\eta-1}}\r)\r)\r]^{\frac{\widetilde{q}}p}
\,\frac{d\lambda}{\lambda}\sim
\int_1^{\infty}{\lambda}^{\widetilde{q}+\frac{\delta}{\eta-1}
\frac{\widetilde{q}}{p}}
\,\frac{d\lambda}{\lambda}
\end{align*}
is also finite if $\widetilde{q}+\frac{\delta}{\eta-1}
\frac{\widetilde{q}}{p}<0$, that is, $\eta>1-\frac{\delta}{p}$.
By $s>\frac{p(\delta-\mu p)}{\delta-p}$, we find that $1-\frac{\delta}{p}
<-\frac{\delta-\mu p}{s}$ is reasonable.
Thus, we could take the parameter $\eta$ such that $1-\frac{\delta}{p}
<\eta\leq-\frac{\delta-\mu p}{s}.$
Therefore, in this case, $\| u_0-b\|_{L^{s,q}(\Omega,\mathcal{H}_{\infty }
^{\delta-\mu p})}=\infty$ for any $b\in\rr$,
but $\|\nabla u_0\|_{L^{p,\widetilde{q}}(\Omega,\mathcal{H}_{\infty }^{\delta })}
<\infty$ for any $\widetilde{q}\in(0,\fz)$,
and hence the exponent $\frac{p(\delta-\mu p)}{\delta-p}$ in
Theorem \ref{them-2}(i) is the best possible. Similarly, one can show that
the exponent $\frac{p(\delta-\mu p)}{\delta-p}$ in
Theorem \ref{them-2}(ii) is also the best possible.
\end{Example}

As a by-product of the proof of Poincar\'e--Sobolev inequalities,
we give the proof of the boundedness of the Riesz potential on
Choquet--Lorentz integrals as follows.

\begin{proof}[Proof of Theorem \ref{them-rp}]
Let $\alpha\in(0,n)$, $\delta\in(0,n]$, and $\mu\in[0,\alpha)$.
To prove (i), suppose $p\in(\frac{\delta}{n},\frac{\delta}{\alpha})$
and $q\in(\frac{\delta(\delta-\mu p)}{n(\delta-\alpha p)},\infty)$.
From Proposition \ref{lem-0908b1}(i), Lemma \ref{lem0909Lc1}, and Theorem \ref{thm0530},
we infer that
\begin{align*}
&\lf\|I_{\alpha}f\r\|_{L^{\frac{p(\delta-\mu p)}{\delta-p\alpha},q}(\rn,
\mathcal{H}_{\infty }^{\delta-\mu p})}\\
&\hs \ls \lf\{\int_0^{\infty}{\lambda}^q\lf[\mathcal{H}_{\infty }
^{\delta-\mu p}\lf(\lf\{x\in\rn:\  \left[\cm_{\mu}(f)(x)\right]
^{\frac{\delta-p\alpha}{\delta-\mu p}}
\|f\|^{\frac{p(\alpha-\mu)}
{\delta-\mu p}}_{L^{p,\frac{q(\delta-p\alpha)}{\delta-\mu p}}
(\rn,\mathcal{H}_{\infty }^{\delta })}>\lambda\r\}\r)\r]
^{\frac{q(\delta-p\alpha)}{p(\delta-\mu p)}}\,\frac{d\lambda}{\lambda}\r\}^{\frac{1}{q}}\\
&\hs\ls \|f\|^{\frac{p(\alpha-\mu)}
{\delta-\mu p}}_{L^{p,\frac{q(\delta-p\alpha)}{\delta-\mu p}}
(\rn,\mathcal{H}_{\infty }^{\delta })}\lf\|\left[\cm_{\mu}(f)\right]
^{\frac{\delta-p\alpha}{\delta-\mu p}}\r\|_{L^{\frac{p(\delta-\mu p)}{\delta-
p\alpha},q}(\rn,\mathcal{H}_{\infty }^{\delta-\mu p})} \\
&\hs\lesssim \|f\|^{\frac{p(\alpha-\mu)}
{\delta-\mu p}}_{L^{p,\frac{q(\delta-p\alpha)}{\delta-\mu p}}
(\rn,\mathcal{H}_{\infty }^{\delta })}
\lf\|\cm_{\mu}(f)\r\|^{\frac{\delta-p\alpha}{\delta-\mu p}}_{L^{p,\frac{q(\delta-p\alpha)}{\delta-\mu p}}(\rn,\mathcal{H}
_{\infty }^{\delta-\mu p})} \\
&\hs\ls\|f\|^{\frac{p(\alpha-\mu)}
{\delta-\mu p}}_{L^{p,\frac{q(\delta-p\alpha)}{\delta-\mu p}}
(\rn,\mathcal{H}_{\infty }^{\delta })}\|f\|^{\frac{\delta-p\alpha}{\delta-\mu p}}
_{L^{p,\frac{q(\delta-p\alpha)}{\delta -\mu p}}(\rn,\mathcal{H}_{\infty }^{\delta })}
\sim\|f\|_{L^{p,\frac{q(\delta-p\alpha)}{\delta-\mu p}}
(\rn,\mathcal{H}_{\infty }^{\delta})},
\end{align*}
which completes the proof of (i).

For (ii), by Lemmas \ref{lem-0908b1}(ii) and \ref{lem0909Lc1} and the weak boundedness
of $\cm_{\mu}$ in Lemma \ref{lem0531} for $p=\frac{\delta}{n}$, we find that
\begin{align*}
&\lf\|I_{\alpha}f\r\|_{L^{\frac{p(\delta-\mu p)}{\delta-p\alpha},\fz}(\rn,
\mathcal{H}_{\infty }^{\delta-\mu p})}\\
&\hs \ls\sup_{\lambda\in(0,\fz)} {\lambda}\lf[\mathcal{H}_{\infty }
^{\delta-\mu p}\lf(\lf\{x\in\rn:\  \left[\cm_{\mu}(f)(x)\right]
^{\frac{\delta-p\alpha}{\delta-\mu p}}
\|f\|^{\frac{p(\alpha-\mu)}
{\delta-\mu p}}_{L^{p}
(\rn,\mathcal{H}_{\infty }^{\delta })}>\lambda\r\}\r)\r]
^{\frac{\delta-p\alpha}{p(\delta-\mu p)}}\\
&\hs\lesssim \|f\|^{\frac{p(\alpha-\mu)}
{\delta-\mu p}}_{L^{p}
(\rn,\mathcal{H}_{\infty }^{\delta })}
\lf\|\cm_{\mu}(f)\r\|^{\frac{\delta-p\alpha}{\delta-\mu p}}_{L^{p,\fz}(\rn,\mathcal{H}
_{\infty }^{\delta-\mu p})} \\
&\hs\ls\|f\|^{\frac{p(\alpha-\mu)}
{\delta-\mu p}}_{L^{p}
(\rn,\mathcal{H}_{\infty }^{\delta })}\|f\|^{\frac{\delta-p\alpha}{\delta-\mu p}}
_{L^{p}(\rn,\mathcal{H}_{\infty }^{\delta })}
\sim\|f\|_{L^{p}
(\rn,\mathcal{H}_{\infty }^{\delta})}.
\end{align*}
Thus, conclusion (ii) is also true.
This then finishes the proof of Theorem \ref{them-rp}.
\end{proof}

We next show that the exponent $\frac{p(\delta-\mu p)}{\delta-p\alpha}$ in Theorem \ref{them-rp} is sharp, that is, replacing this exponent by any $s>\frac{p(\delta-\mu p)}{\delta-p\alpha}$, then Theorem \ref{them-rp} must be false. To construct a counterexample for this, the basic idea comes from Example \ref{ex-0927}, in which the crucial point of that example working is the singularity of $u_0$ near the origin. However, Theorem \ref{them-rp} requires the function to be locally integrable. Thus, the function $u_0$ with singularity near the origin is not permitted. To overcome this difficulty, we construct a sequence of locally integrable functions $\{f_\varepsilon\}_{\varepsilon\in(0,1)}$, whose
limit function as $\varepsilon\to0$ has singularity near the origin.

\begin{Example}\label{rem0619}
For any $\varepsilon\in(0,1)$, let
$f_\varepsilon(x):=|x|^\eta \textbf{1}_{B(\textbf{0},10)\backslash B(\textbf{0},\varepsilon)}(x)$ for any $x\in\rn,$
where $\eta\in(-\fz,0)$ is chosen later.
Then $f_{\varepsilon}\in L^1_{\rm{loc}}(\rn)$ for any $\varepsilon\in(0,\fz)$.
Let $\alpha\in(0,n)$, $\delta\in(0,n]$, $\mu\in[0,\alpha)$, and $p\in(\frac{\delta}{n},\frac{\delta}{\alpha})$.
We claim that, if $s\in(\frac{p(\delta -\mu p)}{\delta-p\alpha},\fz)$,
then there exists $\eta\in(-\fz,0)$ such that, for any $q\in(0,\fz)$,
$\lim_{\varepsilon\to 0}\|I_\alpha (f_\varepsilon)\|_{L^{s,q}(\rn,\ch_\fz^{\delta-\mu p})}=\fz,$
but, for any $\widetilde{q}\in(0,\fz)$,
$\sup_{\varepsilon\in(0,\fz)}\|f_\varepsilon\|_{L^{p,\widetilde{q}}(\rn,\ch_\fz^\delta)}<\fz.$

Indeed, for any $\varepsilon\in(0,1)$ and $x\in \rn$ with $2\varepsilon\le |x|<10\varepsilon$, we have
\begin{align*}
I_\alpha (f_\varepsilon)(x)\ge \int_{B(0,|x|)\backslash B(0,\frac{|x|}2)}\frac{|y|^\eta}{|x-y|^{n-\alpha}}\,dy
\gs |x|^{{\eta+\alpha}}.
\end{align*}
Then, with the assumption $\eta+\alpha<0$, we find that
\begin{align*}
\|I_\alpha (f_\varepsilon)\|_{L^{s,q}(\rn,\ch_\fz^{\delta-\mu p})}^q
&\ge\int_0^\fz\lz^q\lf[\ch_\fz^{\delta-\mu p}\lf(\lf\{x\in B(\textbf{0},10\varepsilon)\backslash B(\textbf{0},2\varepsilon): |x|^{\eta+\alpha}>\lz
\r\}\r)\r]^{\frac qs}\,\frac{d\lz}{\lz}\\
&\ge \int_{(10\varepsilon)^{\eta+\alpha}}^{(2\varepsilon)^{\eta+\alpha}}
\lz^q\lf[\ch_\fz^{\delta-\mu p}
\lf(\lf\{x\in\rn:\ 2\varepsilon\le |x|\le \lz^{1/(\eta+\alpha)}
\r\}\r)\r]^{\frac qs}\,\frac{d\lz}{\lz}\\
&\ge \int_{(10\varepsilon)^{\eta+\alpha}}^{(2\varepsilon)^{\eta+\alpha}}
\lz^q\lf[\lz^\frac{\delta-\mu p}{\eta+\alpha}-(2\varepsilon)^{\delta-\mu p}\r]^{\frac qs}\,\frac{d\lz}{\lz}
\gs \varepsilon^{[\eta+\alpha+(\delta-\mu p)/s]q}\to \fz,
\end{align*}
as $\varepsilon\to0$ if $\eta<-\frac{\delta-\mu p}s-\alpha$. On the other hand, for any $\widetilde{q}$,
by an argument similar to that used in the proof of \eqref{eq0619a}, we conclude that, with $\eta>-\frac\delta p$,
$\|f_\varepsilon\|_{L^{p,\widetilde{q}}(\rn,\ch_\fz^\delta)}\le C$ for some positive constant $C$ independent of
 $\varepsilon$.
By the assumption that $s>\frac{p(\delta-\mu p)}{\delta-p\alpha}$, we find that
$-\frac\delta p<-\frac{\delta-\mu p}{s}-\alpha$. Therefore, it is possible to choose the parameter $\eta$ satisfies
$-\frac\delta p<\eta<-\frac{\delta-\mu p}{s}-\alpha$
and, for any such $\eta$, the above claim holds. This
proves that the exponent $\frac{p(\delta-\mu p)}{\delta-p\alpha}$ in Theorem \ref{them-rp}(i) is the best possible. Similarly, one can show that
the exponent $\frac{p(\delta-\mu p)}{\delta-p\alpha}$ in
Theorem \ref{them-rp}(ii) is also the best possible.
\end{Example}

We end this section by giving a Poincar\'e--Sobolev inequality
for $C^1(\Omega)$-functions with compact support.
First, by \cite[p.\,22, Theorem 2]{MGV1985}, we find that there exists a
positive constant $C$ such that, for any $u\in C^1_{\rm c}(\Omega)$
(the set of all continuously differentiable functions on
$\Omega$ with compact support in $\Omega$) and $x\in\rn$,
\begin{align*}
|u(x)|\leq C\int_{\Omega}\frac{|\nabla u(y)|}{|x-y|^{n-1}}
\,dy,
\end{align*}
which further implies the following conclusion; we omit the details.

\begin{thm} Let $\delta\in\lf(0,n\r]$.
\begin{enumerate}
\item[{\rm (i)}]
If $\Omega\subset\rn$ is a bounded domain and
$p,q\in(\frac{\delta}{n},\infty)$, then there exists a
positive constant $C$, depending only on $n$,
$\delta$, $p$, and $q$, such that, for any $u\in C^1_{\rm c}(\Omega)$,
\begin{align*}
\|u\|_{L^{p,q}(\Omega,\mathcal{H}_{\infty }^{\delta })}\leq
C{\rm diam}(\Omega)~\|\nabla u\|_{L^{p,q}(\Omega,\mathcal{H}_{\infty }^{\delta })}
\end{align*}
and, if $p=\frac{\delta}{n}$, then there exists a
positive constant $C$, depending only on $n$,
$\delta$, and $p$, such that, for any $u\in C^1_{\rm c}(\Omega)$,
\begin{align*}
\|u\|_{L^{p,\fz}(\Omega,\ch_\fz^\delta)}
\le C{\rm diam}(\Omega)~
\|\nabla u\|_{L^{p}(\Omega,\ch_\fz^\delta)}.
\end{align*}
\item[{\rm (ii)}]
If $\Omega\subset\rn$ is a domain, $\mu\in[0,1)$,
$p\in(\frac{\delta}{n},\delta)$, and $q\in(\frac{\delta(\delta-\mu p)}{n(\delta-p)},\infty)$, then there exists
a positive constant $C$, depending only on $n$, $\mu$,
$\delta$, $p$, and $q$, such that, for any $u\in C^1_{\rm c}(\Omega)$,
\begin{align*}
\|u\|_{L^{\frac{p(\delta-\mu p)}{\delta-p},q}(\Omega,
\mathcal{H}_{\infty }^{\delta-\mu p })}\leq
C\|\nabla u\|_{L^{p,\frac{q(\delta-p)}{\delta-\mu p}}
(\Omega,\mathcal{H}_{\infty }^{\delta })}
\end{align*}
and, if $p=\frac{\delta}{n}$, then there exists a
positive constant $C$, depending only on $n$, $\mu$,
$\delta$, and $p$, such that, for any $u\in C^1_{\rm c}(\Omega)$,
\begin{align*}
\|u\|_{L^{\frac{p(\delta-\mu p)}{\delta-p},\fz}
(\Omega,\mathcal{H}_{\infty }^{\delta-\mu p })}
\leq C \|\nabla u\|_{L^{p}(\Omega,\mathcal{H}_{\infty }^{\delta })}.
\end{align*}
\end{enumerate}
\end{thm}



\medskip

\noindent Long Huang

\medskip

\noindent School of Mathematics and Information Science,
Guangzhou University, Guangzhou, 510006, The People's Republic of China

\smallskip

\noindent {\it E-mail}: \texttt{longhuang@gzhu.edu.cn}

\medskip
	
\medskip
	
\noindent Yuanshou Cao and Ciqiang Zhuo (Corresponding author)
	
\smallskip
	
\noindent Key Laboratory of Computing and Stochastic Mathematics
(Ministry of Education), School of Mathematics and Statistics,
Hunan Normal University,
Changsha, Hunan 410081, The People's Republic of China
	
\smallskip
	
\noindent {\it E-mails}:
\texttt{yscao@hunnu.edu.cn} (Y. Cao)
	
\noindent\phantom{{\it E-mails:}}
\texttt{cqzhuo87@hunnu.edu.cn} (C. Zhuo)
	
\medskip
	
\medskip
	
\noindent Dachun Yang
	
\smallskip
	
\noindent Laboratory of Mathematics and Complex Systems
(Ministry of Education of China), School of
Mathematical Sciences, Beijing Normal University, Beijing 100875,
The People's Republic of
China

\smallskip
	
\noindent {\it E-mail}:
\texttt{dcyang@bnu.edu.cn}
	

\begin{thebibliography}{10}
		
\bibitem{ad75}
D. R. Adams, A note on Riesz potentials, Duke Math. J. 42 (1975), 765--778.
	
\vspace{-0.3cm}
		
\bibitem{fz01}
D. R. Adams, A Note on Choquet Integrals with respect to Hausdorff Capacity,
in: Function Spaces and Applications (1986), pp. 115--124,
Lecture Notes in Math. 1302, Springer, Berlin, 1988.
		
\vspace{-0.3cm}
				
\bibitem{su09}
D. R. Adams, Choquet integrals in potential theory, Publ. Mat. 42 (1998), 3--66.
						
\vspace{-0.3cm}
				
\bibitem{agt24}
N. J. Alves, L. Grafakos and A. E. Tzavaras,
A bilinear fractional integral operator for Euler--Riesz systems,
arXiv: 2409.18309.

\vspace{-0.3cm}
		
\bibitem{actuv13}
K. Astala, A. Clop, X. Tolsa, I. Uriarte-Tuero and J. Verdera,
Quasiconformal distortion of Riesz capacities and Hausdorff measures in the plane,
Amer. J. Math. 135 (2013), 17--52.
		
\vspace{-0.3cm}
		
\bibitem{bs1988}
C. Bennett and R. Sharpley, Interpolation of Operators,
Pure and Applied Mathematics 129,
Academic Press, Inc., Boston, MA, 1988.
	
\vspace{-0.3cm}

\bibitem{bl76}
J. Bergh and J. L\"ofstr\"om,
Interpolation Spaces. An introduction, Springer-Verlag, Berlin--New York, 1976.
			
\vspace{-0.3cm}

\bibitem{Boj88}	
B. Bojarski, Remarks on Sobolev imbedding inequalities, in:
Complex Analysis (1987), pp. 52--68,
Lecture Notes in Math. 1351, Springer, Berlin, 1988.

\vspace{-0.3cm}
				
\bibitem{bk91}
Y. Brudnyi and N. Krugljak, Interpolation Functors and Interpolation Spaces,
 North-Holland Publishing Co., Amsterdam, 1991.

				
\vspace{-0.3cm}
				
\bibitem{dhr09}
M. Carro, J. Raposo and J. Soria, Recent Developments in the
Theory of Lorentz Spaces and Weighted Inequalities,
Mem. Amer. Math. Soc. 187 (2007), no. 877, xii+128pp.

\vspace{-0.3cm}

\bibitem{cms11}
J. Cerd\`{a}, J. Mart\'{\i}n and P. Silvestre, Capacitary function spaces,
Collect. Math.
62 (2011), 95--118.

\vspace{-0.3cm}
		
\bibitem{dh11}
G. Choquet, Theory of capacities, Ann. Inst. Fourier (Grenoble)
5 (1953/1954), 13--295.
					
\vspace{-0.3cm}
				
\bibitem{dlyyz}
F. Dai, X. Lin, D. Yang, W. Yuan and Y. Zhang,
Poincar\'e inequality meets Brezis--Van Schaftingen--Yung
formula on metric measure spaces,
J. Funct. Anal. 283 (2022), Paper No. 109645, 52 pp.	
			
		
\vspace{-0.3cm}
				
\bibitem{EG1992}
L. C. Evans and R. F. Gariepy, Measure Theory and Fine Properties of Functions,
CRC Press, Boca	Raton, 1992.

\vspace{-0.3cm}

\bibitem{f76}
H. Federer, Geometric Measure Theory, Springer-Verlag New York, Inc.,
New York, 1969, xiv+676 pp.		
		
\vspace{-0.3cm}
				
\bibitem{ll02}
L. Grafakos, Classical Fourier Analysis, Second edition,
Graduate Texts in Mathematics 249, Springer, New York, 2014.


\vspace{-0.3cm}
				
\bibitem{gbook}
L. Grafakos, Modern Fourier Analysis, Third edition,
Graduate Texts in Mathematics 250, Springer, New York, 2014.
				
\vspace{-0.3cm}
				
\bibitem{hh23}
P. Harjulehto and R. Hurri-Syrj\"{a}nen, On Choquet integrals and
Poincar\'{e}--Sobolev inequalities,
J. Funct. Anal. 284 (2023), Paper No. 109862, 18 pp.

\vspace{-0.3cm}
				
\bibitem{hharx}
P. Harjulehto and R. Hurri-Syrj\"{a}nen, On Sobolev inequalities
with Choquet integrals, arXiv: 2311.09964.
		
\vspace{-0.3cm}
				
\bibitem{hharx1}
P. Harjulehto and R. Hurri-Syrj\"{a}nen, On Choquet integrals
and pointwise estimates, arXiv: 2311.04626.
						

\vspace{-0.3cm}
				
\bibitem{hnsh23}
N. Hatano, T. Nogayama, Y. Sawano and D. I. Hakim, Bourgain--Morrey spaces and their applications to boundedness of operators, J. Funct. Anal. 284 (2023), Paper No. 109720, 52 pp.

\vspace{-0.3cm}
				
\bibitem{HL1972}
L. I. Hedberg, On certain convolution inequalities,
Proc. Amer. Math. Soc. 36 (1972), 505--510.	

\vspace{-0.3cm}
	
\bibitem{liu16}
L. Liu, Hausdorff content and the Hardy--Littlewood maximal
operator on metric measure spaces, J. Math. Anal. Appl. 443 (2016), 732--751.
		
\vspace{-0.3cm}
				
\bibitem{ah15}
Y. Jiao, Y. Zuo, D. Zhou and L. Wu, Variable Hardy--Lorentz
spaces $H^{p(\cdot),q}(\rn)$, Math. Nachr. 292 (2019), 309--349.
				
\vspace{-0.3cm}
				
\bibitem{kr91}
F. John, Rotation and strain, Commun. Pure Appl. Math. 14 (1961), 391--413.
		
\vspace{-0.3cm}
				
\bibitem{Lyw2016}
J. Liu, D. Yang and W. Yuan, Anisotropic Hardy--Lorentz spaces and
their applications, Sci. China Math. 59 (2016), 1669--1720.
							
\vspace{-0.3cm}
				
\bibitem{am05}
G. G. Lorentz, Some new functional spaces, Ann. of Math. (2) 51 (1950), 37--55.
				
\vspace{-0.3cm}
			
\bibitem{L51}
G. G. Lorentz, On the theory of spaces $\Lambda$, Pacific J. Math. 1 (1951), 411--429.
					
\vspace{-0.3cm}
				
\bibitem{Tr92}
O. Martio, John domains, bilipschitz balls and Poincar\'{e} inequality,
Rev. Roum. Math. Pures Appl. 33 (1988), 107--112.
					
\vspace{-0.3cm}
				
\bibitem{MGV1985}
V. G. Maz'ja, Sobolev Spaces, Springer Series in Soviet Mathematics,
Springer-Verlag, Berlin, 1985.
				
\vspace{-0.3cm}
				
\bibitem{n01}
E. Nakai,
On generalized fractional integrals, Taiwanese J. Math. 5 (2001), 587--602.

\vspace{-0.3cm}
				
\bibitem{n10}
E. Nakai, Singular and fractional integral operators on Campanato spaces with variable
growth conditions, Rev. Mat. Complut. 23 (2010), 355--381.

\vspace{-0.3cm}
				
\bibitem{n14}
E. Nakai, Generalized fractional integrals on generalized Morrey spaces,
Math. Nachr. 287 (2014), 339--351.

\vspace{-0.3cm}
				
\bibitem{n17}
E. Nakai,
Singular and fractional integral operators on preduals of Campanato spaces with
variable growth condition, Sci. China Math. 60 (2017), 2219--2240.


\vspace{-0.3cm}
				
\bibitem{cruz03}
N. G. Meyers, A theory of capacities for potentials of functions in
 Lebesgue classes, Math. Scand. 26 (1970), 255--292.
		
\vspace{-0.3cm}
				
\bibitem{hu03}
J. Orobitg and J. Verdera, Choquet integrals, Hausdorff content and
the Hardy--Littlewood maximal operator,
Bull. Lond. Math. Soc. 30 (1998), 145--150.

\vspace{-0.3cm}

\bibitem{ps20}
A. C. Ponce and D. Spector,
A boxing inequality for the fractional perimeter,
Ann. Sc. Norm. Super. Pisa Cl. Sci. (5) 20 (2020), 107--141.


\vspace{-0.3cm}
				
\bibitem{ns12}
Yu. G. Reshetnyak, Integral representations of differentiable
functions in domains with nonsmooth boundary,
Sibirsk. Mat. Zh. 21 (1980), 108--116, 221.
		
%
				
\vspace{-0.3cm}
		
\bibitem{st22}
H. Saito and H. Tanaka,
Dual of the Choquet spaces with general Hausdorff content,
Studia Math. 266 (2022), 323--335.
		
\vspace{-0.3cm}
		
\bibitem{stw20}
H. Saito, H. Tanaka and T. Watanabe,
Block decomposition and weighted Hausdorff content,
Canad. Math. Bull. 63 (2020), 141--156.
		
		
\vspace{-0.3cm}
		
\bibitem{stw19}
H. Saito, H. Tanaka and T. Watanabe,
Fractional maximal operators with weighted Hausdorff content,
Positivity 23 (2019), 125--138.
				
\vspace{-0.3cm}
		
\bibitem{stw16}
H. Saito, H. Tanaka and T. Watanabe,
Dyadic cubes, maximal operators and Hausdorff content,
Bull. Sci. Math. 140 (2016), 757--773.

\vspace{-0.3cm}

\bibitem{sp19}
D. Spector,
A noninequality for the fractional gradient,
Port. Math. 76 (2019), 153--168.

\vspace{-0.3cm}

\bibitem{sv19}
D. Spector and J. Van Schaftingen,
Optimal embeddings into Lorentz spaces for some vector
differential operators via Gagliardo's lemma,
Atti Accad. Naz. Lincei Rend. Lincei Mat. Appl.
30 (2019), 413--436.

\vspace{-0.3cm}

\bibitem{t11}
L. Tang,
Choquet integrals, weighted Hausdorff content and maximal operators,
Georgian Math. J. 18 (2011), 587--596.

\vspace{-0.3cm}
		
\bibitem{v86}
J. Verdera, BMO rational approximation and one-dimensional
Hausdorff content, Trans. Amer. Math. Soc. 297 (1986), 283--304.
		
\vspace{-0.3cm}
				
\bibitem{Ha96}  D. Yang and W. Yuan, A note on dyadic Hausdorff capacities,
Bull. Sci. Math. 132 (2008), 500--509.
		
\vspace{-0.3cm}
				
\bibitem{zyy22}
J. Zhang, D. Yang and S. Yang, Weighted variable Lorentz
regularity for degenerate elliptic equations in Reifenberg domains,
Math. Methods Appl. Sci. 45 (2022), 8487--8502.
		
\vspace{-0.3cm}
			
\bibitem{zsy16}
W. P. Ziemer, Weakly Differentiable Functions, Graduate Texts in
Mathematics 120, Springer-Verlag, New York, 1989.	
	
\end{thebibliography}
\end{document}